\documentclass[11pt]{amsart}
\usepackage{fullpage}

\usepackage{amsmath, amsthm, amssymb, graphicx, upgreek}

\newcounter{citedtheorems}

\newtheorem{defn}{Definition}[section]
\newtheorem{theorem}[defn]{Theorem}
\newtheorem*{theorem-m}{Theorem \ref{main-theorem}}
\newtheorem*{thm-m}{Main Theorem}
\newtheorem*{theorem-abs1}{Theorem \ref{ind-theorem}}
\newtheorem*{theorem-abs2}{Theorem \ref{a23}}
\newtheorem*{theorem-abs3}{Theorem \ref{ind-new}}
\newtheorem*{theorem-abs4}{Theorem \ref{m1}}
\newtheorem*{defn-abs}{Definition \ref{d:excellent}}
\newtheorem{thm-lit}[citedtheorems]{Theorem}
\newtheorem{defn-lit}[citedtheorems]{Definition}
\newtheorem{fact-lit}[citedtheorems]{Fact}
\newtheorem{fact}[defn]{Fact}
\newtheorem{cor}[defn]{Corollary}

\newtheorem{concl}[defn]{Conclusion}
\newtheorem{conv}[defn]{Convention}
\newtheorem{claim}[defn]{Claim}
\newtheorem{lemma}[defn]{Lemma}
\newtheorem{obs}[defn]{Observation}
\newtheorem{rmk}[defn]{Remark}

\newtheorem{disc}[defn]{Discussion}

\newtheorem{expl}[defn]{Example}

\newtheorem{qst}[defn]{Question}

\newcommand{\br}{\vspace{2mm}}
\newcommand{\step}{\br\noindent\emph}
\newcommand{\eff}{\mathcal{F}}
\newcommand{\gee}{\mathcal{G}}

\newcommand{\kleq}{\trianglelefteq}

\newcommand{\tlf}{\trianglelefteq}

\newcommand{\rn}{\operatorname{Range}}
\newcommand{\mpp}{\mathfrak{P}}

\newcommand{\mct}{\mathcal{T}}

\newcommand{\de}{\mathcal{D}}

\newcommand{\fss}{{\mathcal{P}}_{\aleph_0}}
\newcommand{\Los}{\L o\'s }

\newcommand{\trv}{\mathbf{t}} %% the truth value
\newcommand{\uu}{\mathcal{U}}

\newcommand{\jj}{\mathbf{j}}
\newcommand{\mcp}{\mathcal{P}}
\newcommand{\xpu}{{\overline{x}}_{\mcp(u)}}

\newcommand{\ba}{\mathfrak{B}}
\newcommand{\qro}{{\operatorname{Qr}_0}}
\newcommand{\qri}{{\operatorname{Qr}_1}}
\newcommand{\lao}{[\lambda]^{<\aleph_0}}
\newcommand{\lba}{\Lambda_{\ba, \overline{a}}}
\newcommand{\ee}{\mathcal{E}}

\newcommand{\vp}{\varphi}
\newcommand{\lcf}{\operatorname{lcf}}

\newcommand{\ma}{\mathbf{a}}
\newcommand{\mb}{\mathbf{b}}
\newcommand{\mc}{\mathbf{c}}
\newcommand{\mx}{\mathbf{x}}
\newcommand{\my}{\mathbf{y}}

\newcommand{\fin}{\operatorname{Fin}}
\newcommand{\fins}{\operatorname{Fin_s}}
\newcommand{\trg}{T_{rg}}
\newcommand{\dm}{\operatorname{dom}}
\newcommand{\rstr}{\upharpoonright}

\title{A dividing line within simple unstable theories}

\author{M. Malliaris and S. Shelah}\thanks{\emph{Thanks:} 
Malliaris was partially supported by NSF grant DMS-1001666, by a G\"odel fellowship 
and by Shelah's NSF grants DMS-0600940 and DMS-1101597 (Rutgers). 
Shelah was partially supported by Israel Science Foundation grant 710/07. 
This is paper 999 in Shelah's list of publications.}

\address{Department of Mathematics, University of Chicago, 5734 S. University Avenue, Chicago, IL 60637, USA and
Einstein Institute of Mathematics, Edmond J. Safra Campus, Givat Ram, 
The Hebrew University of Jerusalem, Jerusalem, 91904, Israel}
\email{mem@math.uchicago.edu}

\address{Einstein Institute of Mathematics, Edmond J. Safra Campus, Givat Ram, The Hebrew
University of Jerusalem, Jerusalem, 91904, Israel, and Department of Mathematics,
Hill Center - Busch Campus, Rutgers, The State University of New Jersey, 110
Frelinghuysen Road, Piscataway, NJ 08854-8019 USA}
\email{shelah@math.huji.ac.il}
\urladdr{http://shelah.logic.at}

\begin{document}

\begin{abstract}
We give the first (ZFC) dividing line in Keisler's order among the unstable theories, 
specifically among the simple unstable theories.
That is, for any infinite cardinal $\lambda$ for which there is $\mu < \lambda \leq 2^\mu$, 
we construct a regular ultrafilter $\de$ on $\lambda$ such that (i)
for any model $M$ of a stable theory or of the random graph, $M^\lambda/\de$ is $\lambda^+$-saturated but 
(ii) if $Th(N)$ is not simple or not low then $N^\lambda/\de$ is not $\lambda^+$-saturated. 
The non-saturation result relies on the notion of flexible ultrafilters. To prove the saturation result
we develop a property of a class of simple theories, called $\qri$, generalizing the fact that whenever $B$ is a set of parameters
in some sufficiently saturated model of the random graph, $|B| = \lambda$ 
and $\mu < \lambda \leq 2^\mu$, then there is a set $A$ with $|A| = \mu$
such that any nonalgebraic $p \in S(B)$ is finitely realized in $A$. 
In addition to giving information about simple unstable theories,
our proof reframes the problem of saturation of ultrapowers in several key ways. 
We give a new characterization of good filters in terms of ``excellence,'' a measure of the accuracy of the
quotient Boolean algebra. We introduce and develop the notion of {moral} ultrafilters on Boolean algebras.  
We prove a so-called ``separation of variables'' result which shows how the problem of constructing ultrafilters to have
a precise degree of saturation may be profitably separated into a more set-theoretic stage, building an excellent filter,
followed by a more model-theoretic stage: building moral ultrafilters on the quotient Boolean algebra, a process
which highlights the complexity of certain patterns, arising from first-order formulas, in certain Boolean algebras.
\end{abstract}

\maketitle

\section{\hspace{2mm}Introduction}

In 1967 Keisler introduced a framework for comparing the complexity of (countable) first-order theories in terms of the relative difficulty
of producing saturated regular ultrapowers. 
Morley, reviewing the paper \cite{keisler}, wrote that ``the exciting fact is that $\tlf$ gives a rough measure of the `complexity' of a theory. 
For example, first order number theory is maximal while theories categorical in uncountable powers are minimal.'' %\cite{mo}

The only known classes in Keisler's order appear in Theorem \ref{k-known} below.

\begin{thm-lit} \emph{(Results on classes in Keisler's order)} \label{k-known} It is known that:

\[ \mct_1 < \mct_2 < T_{rg} \cdots ?? \cdots \leq \mct_{max} \]
where $\mct_1 \cup \mct_2$ is precisely the class of countable stable theories, and:
\begin{itemize}
\item $T_1$, the minimum class, is the set of all $T$ without the finite cover property $($so stable$)$, e.g. algebraically closed fields.
\item $T_2$, the next largest class, is the set of all stable $T$ with the finite cover property.
\end{itemize}
Among the unstable theories, it is known that:
\begin{itemize}
\item There is a minimum class $\mct_{min}$, which contains the random graph.
\item Among the theories with $TP_2$, there is a minimum class $\mct_{*}$, which contains $T_{feq}$.
\item There is a maximum class $\mct_{max}$, containing all theories with $SOP_3$, thus all linear orders. 
\end{itemize}
However, no model-theoretic identities of unstable classes are known. %
\end{thm-lit}

Keisler in the fundamental 1967 paper \cite{keisler} defined the order and showed the existence of minimum and maximum classes, and defined the finite cover property.
Shelah 1978 \cite{Sh:a} established the identity of $T_1$ and $T_2$, showing that Keisler's order independently detected 
certain key dividing lines from classification theory. Shelah also proved that the theory of linear order, and more generally
$SOP$, belongs to $T_{max}$, and in 1996 \cite{Sh500} extended this to $SOP_3$. Malliaris 2010 \cite{mm4} proved
the existence of a minimum class among the $TP_2$ theories, which contains the theory of a parametrized family 
of independent equivalence relations. 
For further details, see the introduction to the authors' paper \cite{mm-sh-v1}. 
The only prior non-ZFC result was Shelah's result \cite{Sh:c} VI.3.10, which implies 
that if MA and not CH the random graph is not $\aleph_1$-maximal in Keisler's order.

Recently, there has been substantial progress in understanding the interaction of ultrafilters and theories
(Malliaris \cite{mm-thesis}-\cite{mm5}, Malliaris and Shelah \cite{mm-sh-v1}-\cite{treetops}). These results set the stage for our current work.

In this paper we give the first ZFC dividing line among the unstable theories, specifically among the simple unstable theories. 
Our proof reframes the problem of saturation of ultrapowers in terms of so-called excellence of an intermediate filter and so-called morality
of an ultrafilter on the resulting quotient Boolean algebra. We also introduce a property $\qri$ which captures
relevant structure of a class of simple theories including the random graph.

\begin{thm-m} \emph{(Theorem \ref{main-theorem} below)} 
Suppose $\lambda, \mu$ are given with $\mu < \lambda \leq 2^\mu$. Then 
there is a regular ultrafilter $\de$ on $\lambda$ which saturates ultrapowers of
all countable stable theories and of the random graph, but fails to saturate ultrapowers of any non-low or non-simple theory. 
\end{thm-m}

The organization of the paper is as follows. \S 2 is an extended overview of our methods and results.
\S\ref{s:defn} defines Keisler's order, as well as regular, good, flexible (said of filters) and
simple and low (said of theories). \S\ref{s:excellent} motivates and defines the notion of \emph{excellence} for a filter. 
\S\ref{s:separation} defines moral ultrafilters and proves
the theorem on separation of variables. \S\ref{s:l-e} contains the ingredients needed for proving the existence of excellent filters admitting specified homomorphisms. 
It begins with review of constructions via independent families, introduces some notation needed for the current setting 
and concludes by proving the two key inductive steps for the existence proof.
\S\ref{s:existence} contains the existence proof. \S\ref{s:flexible} contains the proof of non-saturation
via non-flexibility described above. \S\ref{s:rg} motivates and defines $\qri = \qri(T, \lambda, \mu)$, 
and proves that this holds for the theory of the random graph. 
\S\ref{s:morality} proves that when $\mu < \lambda \leq 2^\mu$ there is an ultrafilter on $\ba_{2^\lambda, \mu}$ which is moral for 
any theory such that $\qri(T,\lambda, \mu)$. \S\ref{s:dividing-line} contains the statement and proof of the main theorem. 

\br
A continuation of this paper is in preparation. 

\setcounter{tocdepth}{1}
\tableofcontents

\section{\hspace{2mm}Summary of results} \label{s:overview}

To explain our methods, we give here an overview of the main theorems and objects of the paper.
This discussion is informal, and many definitions are deferred to later sections. 

\begin{conv}
For transparency, all filters are regular $($Definition $\ref{regular})$, and all theories are countable. 
\end{conv}

\subsection{Excellent filters.} 
A key point of leverage in our argument is the development of so-called excellent filters, a notion which gives a new
characterization of goodness, Theorem \ref{t:equivalent}. 

To frame our approach, we briefly discuss good filters. A filter $\de$ on $\lambda$ is said to be $\lambda^+$-good if
every monotonic $f: \fss(\lambda) \rightarrow \de$ has a multiplicative refinement, Definition \ref{d:good} below.
From the point of view of saturation, good ultrafilters are maximally powerful in the sense that if $\de$ is a regular
good (i.e. $\lambda^+$-good) ultrafilter on $\lambda$ and $M$ any model in a countable signature, then
$M^\lambda/\de$ is $\lambda^+$-saturated. Moreover, the maximum class in Keisler's order has a set-theoretic characterization:
it is precisely the set of countable theories $T$ such that $M \models T$ and $\de$ regular implies
$M^\lambda/\de$ is $\lambda^+$-saturated iff $\de$ is good, Keisler \cite{keisler}.  [For an account of this correspondence, 
relevant history and recent work on Keisler's order, see \cite{mm-sh-v1} Sections 1, 4.]
The existence of $\lambda^+$-good ultrafilters on $\lambda$ is a result of Keisler assuming GCH \cite{keisler-1} 
and of Kunen in ZFC \cite{kunen}. See for instance \cite{Sh:c} pps. 345-347.

To make finer distinctions in ultrafilter construction, one needs a greater degree of precision than is a priori available from the
definition of goodness. The issue comes into focus when working with {filters} rather than ultrafilters, as we now discuss.

\step{Remark 1.}   Let $\de$ be a regular good filter on $I$ and let $\overline{A} = \langle A_u : u \in \lao \rangle$ be a monotonic sequence of nonzero elements of $\mcp(I)$, not necessarily a sequence of elements of $\de$. A priori, goodness does not guarantee
a multiplicative refinement for $\overline{A}$. Moreover, suppose the image of $\langle A_u : u \in \lao \rangle \subseteq \mcp(I)$ in the quotient Boolean algebra
$\ba = \mcp(I)/\de$ is $\langle \ma_u : u \in \lao \rangle$, and $\overline{\ma}$ has a multiplicative refinement $\langle \mb_u : u \in \lao \rangle$ in $\ba$.
Then $\overline{A}$ will have a refinement $\overline{B}$ which is multiplicative \emph{mod $\de$}. That is, $u,v \in \lao \implies B_u \cap B_v = B_{u \cup v} \mod \de$, which a priori does not imply that an actual
multiplicative refinement of $\overline{A}$ exists.

\step{Remark 2.} From this and other considerations, one sees that a useful intensification of goodness will be
\emph{making the quotient Boolean algebra more precise}. What does this mean? 
Roughly speaking, that properties of sequences in the quotient Boolean algebra
accurately reflect those in $\mcp(I)$: if a sequence in the quotient Boolean algebra appears to have
certain properties, e.g. multiplicativity, then we can indeed pull it back to a multiplicative sequence in $\mcp(I)$. 

\step{Remark 3}. 
The right notion of ``properties'' is suggested by the slogan ``solving equations modulo $\de$.''  That is, we will 
be concerned with preserving certain distinguished terms in the language of Boolean algebras. (The definition of this set of terms
$\Lambda$, Definition \ref{d:near} below, arises naturally from the inductive construction of excellence, see 
Claim \ref{e-ind-step}, and also \ref{e:boolean}.)

Indeed, the form of this abstraction naturally suggests that we can, at little cost, build $\de$ to be accurate for a range of properties
including, but not limited to, multiplicativity; see e.g. \ref{c:some-examples}.

\begin{defn-abs} \emph{(Excellent filters)} %\label{d:excellent} 
\\Let $\de$ be a filter on the index set $I$. We say that $\de$ is \emph{$\lambda^+$-excellent}
when: if $\overline{A} = \langle A_u : u \in \lao \rangle$ with $u \in \lao \implies A_u \subseteq I$, then
we can find $\overline{B} = \langle B_u : u \in \lao \rangle$ such that:
\begin{enumerate}
\item for each $u \in \lao$, $B_u \subseteq A_u$
\item for each $u \in \lao$, $B_u = A_u \mod \de$
\item \emph{if} $u \in \lao$ and $\sigma \in \Lambda = \Lambda_{\de, \overline{A}|_u}$, so $\sigma(\overline{A}|_{\mcp(u)}) = \emptyset \mod \de$, 
\emph{then} $\sigma(\overline{B}|_{\mcp(u)}) = \emptyset$
\end{enumerate}
We say that $\de$ is $\xi$-excellent when it is $\lambda^+$-excellent for every $\lambda < \xi$. 
\end{defn-abs}

\step{Remark 4.} The analysis and definition of excellence will have the following consequences for ultrafilter construction. Once we have a notion of excellent filter, there is a potentially two-stage construction of a given ultrafilter in which we 
first ensure excellence of some intermediate filter $\de$, 
and then move to work directly on the quotient Boolean algebra for the remainder of the
construction, leveraging the guarantee that the work on the Boolean algebra will be sufficiently accurate.
 
This ``paradigm shift'' is accomplished in Section \ref{s:separation} with Theorem \ref{t:separation}, also quoted later in this
introduction. We have called Theorem \ref{t:separation} a separation of variables result. 
In some sense it allows us to separate the more
set-theoretic considerations involved in building an excellent filter from the more model-theoretic considerations 
involving the complexity of patterns coming from a given formula in certain Boolean algebras $\ba_{2^\lambda, \mu}$.

\br

\subsection{Morality and separation of variables} 
The phenomenon of excellence naturally gives rise to a complementary property we call ``morality,'' Definition \ref{d:moral}.
Say that an ultrafilter $\de_*$ on a Boolean algebra $\ba$ is moral for a theory $T$ if, roughly speaking, 
any incidence pattern for $T$ represented in $\ba$ can be resolved (multiplicatively refined) in $\de_*$.  
The definition 
does not rely on the setting of reduced products.  
Excellence and morality then combine to give saturation in the following way:

\br
\noindent\textbf{Theorem \ref{t:separation}. (Separation of variables)}
\emph{Suppose that we have the following data:}
\begin{enumerate}
\item \emph{$\de$ is a regular, $\lambda^+$-excellent filter on $I$}
\item \emph{$\de_1$ is an ultrafilter on $I$ extending $\de$, and is ${|T|}^+$-good}
\item \emph{$\ba$ is a Boolean algebra }
\item \emph{$\jj : \mcp(I) \rightarrow \ba$ is a surjective homomorphism with $\de = \jj^{-1}(\{ 1_\ba \})$}
\item \emph{$\de_* = \{ \mb \in \ba : ~\mbox{if} ~\jj(A) = \mb ~\mbox{then}~ A \in \de_1 \}$}
\end{enumerate}
\emph{Then the following are equivalent:}
\begin{itemize}
\item[(A)] \emph{$\de_*$ is $(\lambda, \ba, T)$-moral, i.e. moral for each formula $\vp$ of $T$.}
\item[(B)] \emph{For any $\lambda^+$-saturated model $M \models T$, $M^\lambda/\de_1$ is $\lambda^+$-saturated.}
\end{itemize}
\br

\begin{disc}
What does Theorem \ref{t:separation} accomplish? At first glance, it may appear that we have traded one construction problem
(building an ultrafilter on $\lambda$) for another (building an ultrafilter on $\ba$). The gain is revealing a point of 
leverage which will allow us to separate theories by realizing some types while omitting others. 
The leverage is provided by the size of a maximal antichain in the quotient Boolean algebra $B(\de) = \mcp(I)/\de$.
By Theorem \ref{t:existence}, when building excellent filters we are relatively free to modify this quotient Boolean algebra.

Then the strategy is as follows. 
For non-saturation, we will show that if $CC(B(\de)) = \mu^+ < \lambda^+$, no subsequent ultrafilter can be flexible, and then apply our prior work. 
For saturation, we will show that when $\mu < \lambda \leq 2^\mu$ this need not prevent saturation of the random graph.
\end{disc}

\br

Sections \ref{s:l-e}-\ref{s:existence} contain a proof of the existence of excellent filters meeting the requirements
of Theorem \ref{t:separation}. Theorem \ref{t:existence}, which we now quote, is more general than what is needed for the
application to Theorem \ref{main-theorem}. In that specific case, one could use Theorem \ref{t:equivalent} showing the equivalence
of excellent and good, and then build $\de$ to be a good regular filter on $\lambda$ such that $(I, \de, \gee)$ 
is a $(\lambda, \mu)$-good triple. The more general result reflects the fact that the framework of 
Theorem \ref{t:separation} is a main contribution of the paper; 
we make significant further use of this framework, for a wider range of Boolean algebras, in work in preparation.   
Moreover, note that the inductive Claim \ref{e-ind-step} of Theorem \ref{t:existence} clearly shows 
the naturalness of the definition of $\Lambda$ from \ref{d:near}, and thus in some sense, its optimality.

\br
\noindent\textbf{Theorem \ref{t:existence}. (Existence theorem)}
\emph{Let $\mu \leq \lambda$, $|I| = \lambda$ and let $\ba$ be a $\mu^+$-c.c. complete Boolean algebra of cardinality $\leq 2^\lambda$. 
Then there exists a regular excellent filter $\de$ on $I$ and a surjective homomorphism $\jj: \mcp(I) \rightarrow \ba$
such that $\jj^{-1}(1) = \de$.}
\br

Theorem \ref{t:existence} requires several lemmas and some notation, but otherwise proceeds smoothly. Briefly, we need to accomplish two things:
first, to ``solve'' all instances of excellence, and second to ensure the existence of the homomorphism $\jj$. 
We begin with a regular filter $\de_0$ and two disjoint independent families, $\eff \subseteq {^I\lambda}$ of cardinality $2^\lambda$,
and $\gee \subseteq {^I2}$ of cardinality $|\ba|$. We extend to a second filter $\de_1$ in which 
$\{ \mb_\gamma : \gamma < |\ba| \}$ and $\{ g_\gamma^{-1}(1) : \gamma < |\gee| \}$ ``look alike'' in the sense of Definition \ref{d:m1}.
We then build the filter $\de$ by induction on $\alpha < 2^\lambda$ while respecting this background correspondence, consuming the functions $\eff$
while giving no further constraints on $\gee$. At odd successor stages, we ensure that a given subset of $I$ will have an appropriate
homomorphic image via Lemma \ref{l:free}. At even stages, we solve instances of excellence using Claim \ref{e-ind-step}.

\subsection{Non-saturation}
The non-saturation results arise via the notion of
\emph{flexible} filter, introduced in Malliaris \cite{mm-thesis}. By Malliaris \cite{mm-thesis}, 
flexibility is a necessary condition for an ultrafilter to saturate some non-low theory. 
By Malliaris \cite{mm4} for the case of $TP_2$, and Malliaris and Shelah \cite{treetops} for the case of $SOP_2$, 
flexibility is a necessary condition for an ultrafilter $\de$ to saturate some non-simple theory. 
We then adapt a proof of Shelah \cite{Sh:c} originally stated for goodness to show that when the c.c. of the quotient Boolean algebra
falls at or below the size of the index set, no subsequent ultrafilter will be flexible, and thus
every subsequent ultrafilter will fail to saturate any non-simple or non-low theory:

\br
\noindent \textbf{Corollary \ref{f:cor}. (Non-flexibility, thus non-saturation)}
\emph{Let $\mu < \lambda$ and let $\de$ be a regular $\lambda^+$-excellent filter on $\lambda$ given by
Theorem \ref{t:existence} in the case where $\jj(\mcp(I)) = \ba_{2^\lambda, \mu}$
$($so has the $\mu^+$-c.c.$)$. Then no ultrafilter extending $\de$, built by the methods of independent functions, is $\lambda$-flexible.}
\br

\subsection{Saturation}
The saturation results arise from a property of the random graph used by Shelah in \cite{Sh:c} Theorem VI.3.10.
This key property, which follows from Engelking-Karlowicz \cite{ek}, 
is that if $\mu < \lambda \leq 2^\mu$,  
$A \subseteq \mathfrak{C}_{\trg}$,
$|A| \leq \lambda$ then for some $B \subsetneq \mathfrak{C}_{\trg}$, $|B|=\mu$ we have that every nonalgebraic
$p \in S(A)$ is finitely realized in $B$. 
That is, the nonalgebraic types over a given set of size $\lambda$ can be finitely realized in a set of strictly smaller size; 
see \S \ref{s:rg} below for a proof.
Note that by \cite{Sh93}, every simple theory $T$ has a related, though weaker, property. 

We develop a generalization of this property appropriate for our context, called $\qri$.
Informally, $\qri(T, \lambda, \mu)$ says of $T$ that any monotonic sequence from $\fss(\lambda)$
into the given Boolean algebra $\ba = \ba_{2^\lambda, \mu}$, which accurately
``represents'' a pattern from the background theory $T$, can be approximated by
$\mu$ multiplicative sequences.  This property may be thought of as describing genericity, in the sense
of the independence property; it is naturally orthogonal, in a non-technical sense, 
to the phenomenon of dividing, in which the many instances of the formula $\vp$
are ``spread out'' and do not admit common realizations. 
In \S \ref{s:morality} we show that $\qri$ holds of the random graph: 

\br
\noindent\textbf{Lemma \ref{l:sequence}. ($\qri$ for the random graph)}
\emph{Let $T$ be the theory of the random graph. Then $\qri(T, \lambda, \mu)$.}
\br

Briefly, to prove Lemma \ref{l:sequence} begin with a ``possibility pattern,'' the avatar of a type. 
Choose a complete subalgebra of $\ba$ on which this sequence is supported,
and which itself is covered by few ultrafilters. Roughly speaking, we look at what happens to the type under each such ultrafilter
(we define a function which records how the parameters collide) 
and choose a small dense subset of types over this ``collapsed'' parameter set. Since types over the
``collapsed'' sets have parameters which are everywhere distinct, they will always be realized.  Now to find a ``cover'' of a
given finite fragment of the original type, we can first choose an ultrafilter in which its finitely many parameters remain distinct, then choose an appropriate member of the dense set of realized types. 

In \S \ref{s:morality} we apply Lemma \ref{l:sequence} to prove the existence of an ultrafilter $\de_*$ on $\ba_{2^\lambda, \mu}$ which is
moral for the theory of the random graph:

\br
\noindent\textbf{Theorem \ref{t:morality}. (The moral ultrafilter)}
\emph{Suppose $\mu < \lambda \leq 2^\mu$ and let $\ba = \ba_{2^\lambda,~\mu}$. Then there is an ultrafilter $\de_*$ on $\ba$ which is moral
for all countable theories $T$ such that $\qri(T, \lambda, \mu)$.  In particular, $\de_*$ is moral for all countable stable theories
and for the theory of the random graph.} 
\br

Finally, we combine these results to prove the main theorem:

\br
\begin{theorem-m} \emph{\textbf{(The dividing line)}}
Let $\mu < \lambda \leq 2^\mu$. Then there is a regular ultrafilter $\de$ on $\lambda$ such that:
\begin{enumerate}
\item for any countable theory $T$ such that $\qri(\lambda, \mu, T)$ and $M\models T$,
$M^\lambda/\de$ is $\lambda^+$-saturated.
\item in particular, when $T$ is stable or $T$ is the theory of the random graph,
$M^\lambda/\de$ is $\lambda^+$-saturated.
\item for any non-low or non-simple theory $T$ and $M \models T$, $M^\lambda/\de$ is not $\lambda^+$-saturated.
\end{enumerate}
Thus there is a dividing line in Keisler's order among the simple unstable theories. 
\end{theorem-m}

\br
\begin{disc} 
Following our work here, to separate theories $T, T^\prime$ in Keisler's order it is therefore sufficient 
to find a pair $(\ba, \de_*)$ s.t. $\ba$ is a $\lambda^+$-c.c. complete 
Boolean algebra of cardinality $\leq 2^\lambda$ and $\de_*$ an ultrafilter on $\ba$ 
which is $(\lambda, T)$-moral but not $(\lambda, T^\prime)$-moral. 
Note that this gives natural new ``outside'' definitions of classes of first-order theories in terms of whether e.g. every ultrafilter
on a given Boolean algebra is moral for $T$. Do such classes have nice inside definitions?
\end{disc}

\section{\hspace{2mm}Basic definitions} \label{s:defn}

Here we define Keisler's order, the properties (of filters) regular and good, and the properties
(of theories) simple and low. A fairly extensive discussion of Keisler's order, including an overview of relevant
recent work \cite{mm-thesis}-\cite{mm-sh-v2}, can be found in Malliaris and Shelah 2011 \cite{mm-sh-v1}. For an account of simplicity, including
results on Boolean algebras of nonforking formulas in simple theories from \cite{Sh93}, see the ``Primer of Simple Theories'' 
of Grossberg, Iovino and Lessmann \cite{GIL}.  
For further background on ultrafilters and ultrapowers, see \cite{Sh:c} Chapter VI, \cite{keisler}, \cite{mm-thesis}.

For transparency, we consider \emph{countable} first-order theories. 
We concentrate on regular ultrafilters which, by Theorem \ref{backandforth} below, focus attention on the theory rather than the choice of index models. 

\begin{conv}
{All ultrafilters in this paper, unless otherwise stated, are regular, Definition \ref{regular}.}

By ``$\de$ saturates $T$'' we will always mean:
$\de$ is a regular ultrafilter on the infinite index set $I$, $T$ is a countable complete first-order theory and for any $M \models T$,
we have that $M^I/\de$ is $\lambda^+$-saturated, where $\lambda = |I|$. 

We generally write ``$\de$ is an ultrafilter on $\lambda$'' thereby specifying the index set $\lambda$, rather 
than the field of sets $\mathcal{P}(\lambda)$. 
\end{conv}

\begin{defn} \label{regular}
A filter $\de$ on an index set $I$ of cardinality $\lambda$ is said to be \emph{$\lambda$-regular}, or simply 
\emph{regular}, if there exists a $\lambda$-regularizing family $\langle X_i : i<\lambda \rangle$, which means that:
\begin{itemize}
 \item for each $i<\lambda$, $X_i \in \de$, and
 \item for any infinite $\sigma \subset \lambda$, we have $\bigcap_{i \in\sigma} X_i = \emptyset$
\end{itemize}
Equivalently, for any element $t \in I$, $t$ belongs to only finitely many of the sets $X_i$. 
\end{defn}

Let $I = \lambda \geq \aleph_0$ and fix $f: \fss(\lambda)\rightarrow I$. Then 
$\{ \{ s\in I : \eta \in f^{-1}(s) \} : \eta < \lambda \}$ can be extended to a regular filter on $I$,
so regular ultrafilters on $\lambda \geq \aleph_0$ always exist, see \cite{ck73}.

By the next theorem, when $T$ is countable and $\de$ is regular, saturation of the ultrapower does not depend on the choice of index model. 
Thus the restriction to regular filters justifies the quantification over all models in Keisler's order, Definition \ref{keisler-order} below.

\begin{thm-lit} \label{backandforth} \emph{(Keisler \cite{keisler} Corollary 2.1 p. 30; see also Shelah \cite{Sh:c}.VI.1)}
Suppose that $M_0 \equiv M_1$, the ambient language is countable, and
$\de$ is a regular ultrafilter on $\lambda$.
Then ${M_0}^\lambda/\de$ is $\lambda^+$-saturated iff ${M_1}^\lambda/\de$
is $\lambda^+$-saturated.
\end{thm-lit}

\begin{defn}
A function with domain $\fss(\kappa)$ is called \emph{monotonic} if $u \subseteq v \in \fss(\kappa)$ implies
$f(v) \subseteq f(u)$, and \emph{multiplicative} if $f(u) \cap f(v) = f(u \cup v)$. 
\end{defn}

From the point of view of saturation, the most powerful ultrafilters are the \emph{good} ultrafilters, introduced
by Keisler \cite{keisler-1}. These saturate any [countable] theory, and thus witness the existence of a maximum class in Keisler's order. 

\begin{defn} \label{d:good}
A filter $\de$ on $\lambda \geq \aleph_0$ is called \emph{$\kappa^+$-good} if 
every monotonic function $f: \fss(\kappa) \rightarrow \de$ has a multiplicative refinement. 
$\de$ is called \emph{good} if it is $\lambda^+$-good.
\end{defn}

Keisler proved the existence of $\lambda^+$-good countably incomplete ultrafilters on $\lambda$ assuming $2^\lambda = \lambda^+$. 
Kunen \cite{kunen} gave a proof in ZFC, using independent families of functions (also called families of large oscillation). 
Kunen's construction technique and its subsequent development by the second author in Chapter VI of \cite{Sh:c} is a key ingredient of our approach in this paper. 

We now give some important model-theoretic properties. 
The reader interested primarily in ultrafilters rather than model theory may take
these properties as black boxes which give the non-saturation side of the argument in \S \ref{s:flexible}.

\begin{defn} \emph{(Simple, low)} \label{defn-low} Given a background theory $T$,
\begin{enumerate}
 \item A formula $\vp$ is \emph{simple} if for every $k < \omega$, $D(x=x,\{\vp\}, k) < \omega$.
 \item A formula $\vp = \vp(x,y)$ is \emph{low} if there exists $k = k_\vp <\omega$ such that $D(x=x, \{\vp\}, \infty) = D(x=x, \{\vp\}, k)$. 
Equivalently, there is $k = k_\vp$ such that for any indiscernible sequence of parameters $\langle a_i : i < \omega \rangle$,
$\ell(a_i) = \ell(y)$, if $\{ \vp(x;a_i) : i < \omega \rangle$ is 1-consistent, i.e. each 1-element subset is consistent, 
then either it is consistent or it is uniformly $k$-inconsistent, i.e. each $k$-element subset is inconsistent. 
 \item $T$ is said to be \emph{simple}, respectively \emph{low}, if every formula of $T$ is. 
 \item A theory which is not low is often called \emph{non-low}.
\end{enumerate} 
\end{defn}

\begin{rmk}
For simple theories, (2) is equivalent to Buechler's original definition 
\emph{\cite{buechler}} which asked that for every $\vp$, $D(x=x, \{\vp\}, \aleph_0) < \omega$. 
\end{rmk}

\br

Finally, we define Keisler's order, proposed in Keisler 1967 \cite{keisler}. 
This preorder $\kleq$ on theories is often thought of as a partial order on the $\kleq$-equivalence classes.
The hypothesis \emph{regular}, Definition \ref{regular}, justifies the quantification over all models.

\begin{defn} \label{keisler-order} \emph{(Keisler \cite{keisler})} 
Given countable theories $T_1, T_2$, say that:
\begin{enumerate}
 \item  $T_1 \kleq_\lambda T_2$ if for any $M_1 \models T_1, M_2 \models
T_2$, and $\de$ a regular ultrafilter on $\lambda$, \\if 
$M^{\lambda}_2/\de$
is $\lambda^+$-saturated then $M^{\lambda}_1/\de$ must be 
$\lambda^+$-saturated.
\item \emph{(Keisler's order)}
$T_1 \kleq T_2$ if for all infinite $\lambda$, $T_1 \kleq_\lambda T_2$.
\end{enumerate}
\end{defn}

\begin{qst}
Determine the structure of Keisler's order.
\end{qst}

\section{\hspace{2mm}Excellent filters} \label{s:excellent}

In this section we define ``$\lambda^+$-excellent filter,'' Definition \ref{d:excellent} below and develop some consequences
of this definition.

\begin{conv} \emph{(Conventions)} \label{c:1}
\begin{itemize}
\item We consider Boolean algebras in the language $\{ \cap, \cup, \leq, -, 0, 1 \}$ 
and will informally use symmetric difference $\Delta$ and setminus $\setminus$.  
$($Note that negation of $\ma$ will be denoted $1-\ma$;  overline, e.g $\overline{\ma}$, always denotes a sequence, not the complement of a set.) 
\item $\ba$ denotes a Boolean algebra, and all elements of Boolean algebras are written in boldface:
$\ma, \mb...$
\item $CC(\ba)$ is the minimum regular cardinal $\mu$ such that any partition (maximal disjoint subset) 
of $\ba$ has cardinality less than $\mu$.
\item For $\de$ a filter on the index set $I$, $B(\de)$ is the Boolean algebra $\mcp(I)/\de$. 
\item When $\de$ is a filter on an index set $I$ (or a Boolean algebra $\ba$),
$\de^+$ denotes the sets which are positive modulo $\de$, i.e. not equal to $\emptyset \mod \de$
\item If $X$ is a formula then we use the shorthand $X^1, X^0$ to denote $X, \neg X$ respectively.
\item $\Delta$ is used both for symmetric difference and for sets of formulas, but this is always clear from context.
\item $\de, \ee$ denote filters. 
\end{itemize}
\end{conv}

\begin{defn} \emph{(Boolean terms)}
\begin{enumerate}
\item Let $u$ be a finite set. We write $\overline{x}_{\mcp(u)} = \langle x_v : v \subseteq u \rangle$
for a sequence of variables indexed by subsets of $u$. 
\item By \emph{Boolean term} we mean a term in the language of Boolean algebras, see Convention $\ref{c:1}$.
\item For a Boolean term $\sigma$, we write $\sigma=\sigma(\xpu)$ to indicate the free variables are indexed this way. 
For $\sigma=\sigma(\xpu)$ a Boolean term and $\langle A_u  : u \in \lao \rangle$ a sequence of elements 
of some given Boolean algebra, we write $\sigma(\overline{A}|_{\mcp(u)})$ or equivalently, 
$\sigma( \langle A_v  : v \subseteq u \rangle )$ for the term evaluated on the relevant part of the sequence.
\end{enumerate}
\end{defn}

We will consider certain distinguished sets of Boolean terms. For further motivation, 
see Example \ref{e:boolean} and Claim \ref{c:some-examples} below.

\begin{conv} \label{c:near}
Let $\ba$ be a Boolean algebra and $\overline{\ma} = \langle \ma_u : u \in \lao \rangle$ be a sequence of 
elements of $\ba$. Let 
\[  N(\overline{a}{\rstr}_{\mcp(u)}) = \{ \langle a^\prime_v : v \subseteq u \rangle : ~\mbox{for some $w \subseteq u$ we have $a^\prime_v = a_v$ if $v \subseteq w$
and $a^\prime_v = 0_\ba$ otherwise} \} \]
\end{conv}

\begin{rmk}
For the purposes of this paper, we are interested in so-called possibility patterns, 
Definition \ref{d:poss} and thus it will be sufficient to restrict to \emph{monotonic}, $\lao$-indexed sequences, 
i.e. $v \subseteq u \implies \ma_u \subseteq \ma_v$, allowing some elements of the sequence to be $0$.
\end{rmk}

\begin{defn} \label{d:near}
For $\ba$ a Boolean algebra, $u$ finite, $\overline{a} = \langle a_v : v \subseteq u \rangle$
a sequence of members of $\ba$, 
\begin{enumerate}
\item Define $\lba$ to be the set 
\[\{ \sigma(\xpu) : \sigma(\xpu) ~\mbox{is a Boolean term such that $\ba \models$ ``$\sigma(\overline{a}^\prime) = 0$''
whenever $\overline{a}^\prime \in N(\overline{a})$ }\} \]

\item If $\de$ is a filter on $\ba$ then 
$\Lambda_{\ba, \de, \overline{a}} = \Lambda_{\ba_1, \overline{a}_1}$ where $\ba_1 = \ba/\de$ and 
$\overline{a}_1 = \langle a_v/\de : v \subseteq u \rangle$. 

\item If $\de$ is a filter on a set $I$, then $\de$ determines $I$, so we write
$\Lambda_{\de, \overline{a}}$ for $\Lambda_{\mcp(I), \de, \overline{a}}$.
\end{enumerate}
\end{defn}

\br

We now give one of the central definitions of the paper:

\begin{defn} \emph{(Excellent filters)} \label{d:excellent}
\\Let $\de$ be a filter on the index set $I$. We say that $\de$ is \emph{$\lambda^+$-excellent}
when: if $\overline{A} = \langle A_u : u \in \lao \rangle$ with $u \in \lao \implies A_u \subseteq I$, then
we can find $\overline{B} = \langle B_u : u \in \lao \rangle$ such that:
\begin{enumerate}
\item for each $u \in \lao$, $B_u \subseteq A_u$
\item for each $u \in \lao$, $B_u = A_u \mod \de$
\item \emph{if} $u \in \lao$ and $\sigma \in \Lambda_{\de, \overline{A}|_u}$, so $\sigma(\overline{A}|_{\mcp(u)}) = \emptyset \mod \de$, 
\emph{then} $\sigma(\overline{B}|_{\mcp(u)}) = \emptyset$
\end{enumerate}
We say that $\de$ is $\xi$-excellent when it is $\lambda^+$-excellent for every $\lambda < \xi$. 
\end{defn}

In this paper we focus on regular excellent filters. 
Since we will often refer to sequences of the kind just described, we give them a name:

\begin{defn}
Given $I$, $\de$, $\overline{A}$ as in Definition \ref{d:excellent}, we will call any $\overline{B}$ satisfying
$(1)$-$(3)$ of Definition \ref{d:excellent} a \emph{$\de$-excellent refinement of $\overline{A}$}. When the identity of 
$\de$ is clear, we may simply say ``excellent refinement.'' 
\end{defn}

\begin{expl} \label{e:boolean} 
The distinguished terms capture equations we can solve by isolating zero-sets which we can safely eliminate. 
As an example of why respecting \emph{all} Boolean terms would be too strong, consider a monotonic sequence $\overline{A} = \langle A_u : u \in \lao \rangle$ of elements of $\de$. 
Then for each $u, v \in \lao$, $A_u = A_v \mod \de$, i.e. $A_u \Delta A_v = \emptyset \mod \de$. 
Asking for a $\de$-equivalent refinement $\overline{B}$ in which $u, v \in \lao$ implies $B_u = B_v$ would 
require an instance of completeness, i.e. $\bigcap \{ A_u : u \in \lao \} \in \de$. 
\end{expl}

We now verify that the cases of main interest are captured by the definition of excellent.

\begin{claim} \label{c:some-examples}
Let $\de$ be a filter on $I$, i.e. on $\mcp(I)$.  
\begin{enumerate}
\item If $\de$ is $\lambda^+$-excellent and $A_u \subseteq I$ for $u \in \lao$,
then we can find $\overline{B}$ such that:
\begin{enumerate}
\item $\overline{B} = \langle B_u : u \in \lao \rangle$
\item $B_u \subseteq A_u$
\item $B_u = A_u \mod \de$
\item for all $u_0, u_1 \in \lao$, if $A_{u_0} \cap A_{u_1} = A_{u_0 \cup u_1} \mod \de$ then 
$B_{u_0} \cap B_{u_1} = B_{u_0 \cup u_1}$
\end{enumerate}

\item If $\de$ is $\lambda^+$-excellent and $A_\alpha \subseteq I$ for $\alpha < \lambda$,
then we can find $\overline{B}$ such that:
\begin{enumerate}
\item $\overline{B} = \langle B_\alpha : \alpha < \lambda \rangle$
\item $B_\alpha \subseteq A_\alpha$
\item $B_\alpha = A_\alpha \mod \de$
\item if $n \in \mathbb{N}$, $\alpha_0, \dots \alpha_{n-1} < \lambda$ and 
$\bigcap \{ A_{\alpha_\ell} : \ell < n \} = \emptyset \mod \de$, then 
$\bigcap \{ B_{\alpha_\ell} : \ell < n \} = \emptyset$
\end{enumerate}

\item If $\de$ is $\lambda^+$-excellent then $\de$ is $\lambda^+$-good.
\end{enumerate}
\end{claim}

\begin{proof}
Note that (3) follows from (1) in the case where the sequence $\langle A_u : u \in \lao \rangle$ is assumed to be a sequence of elements of $\de$.

\br
(1) Let $\overline{A} = \langle A_u : u \in \lao \rangle$ be given, with $A_u \in \de^+$. 
We are assuming $\de$ is $\lambda^+$-excellent, so let  $\overline{B} = \langle B_u : u \in \lao \rangle$ be an excellent refinement. Then conditions (a), (b), (c) 
hold by definition. For condition (d), it will suffice to show that \emph{if} (*) for all $u_0, u_1 \in \lao$,  $A_{u_0} \cap A_{u_1} = A_{u_0 \cup u_1} \mod \de$, 
\emph{then} (**), where $(**)$ is the condition that, writing $u = u_0 \cup u_1$, the Boolean term
\[ \sigma(\xpu) = \left( ( x_{u_0} \cap x_{u_1} ) \Delta x_{u} \right)  \in \Lambda_{\de, \overline{A}|_{\mcp(u)}} \]
Why would this suffice? Because we would then have, as an immediate consequence of ``excellent refinement,'' that
$u_0, u_1 \in \lao$ implies $B_{u_0} \cap B_{u_1} = B_u$ {since} $(**)$ implies $A_{u_0} \cap A_{u_1} = A_u \mod \de$. 

So let us prove that $(*)$ implies $(**)$. 
That is, we verify that $(*)$ implies $\sigma$ evaluates to $\emptyset \mod \de$ on any term ``below'' $\overline{A}|_{\mcp(u)}$ in the sense of Definition \ref{d:near}.
For any $w \subseteq u := u_0 \cup u_1$,
\begin{itemize}
\item If $w = u$, then $\sigma = \emptyset \mod \de$ as $A_{u_0} \cap A_{u_1} = A_{u_2} \mod \de$ by $(*)$. 
\item if $w \subseteq u_1 \land w \not\subseteq u_0$ then we have either $\emptyset \cap \emptyset = \emptyset \mod \de$
or else $A_{u_0} \cap \emptyset = \emptyset \mod \de$, both of which are clearly true;
\item the case $w \subseteq u_1 \land w \not\subseteq u_0$ is the same as the previous case by symmetry.
\end{itemize}

In other words, 
since a sufficient condition for being multiplicative is expressible by one of our distinguished terms,
any sequence $\overline{A}$ which is multiplicative $\mod \de$, even if it does not itself consist of elements of $\de$, will have a true multiplicative refinement 
$\overline{B}$ as desired.

This completes the proof.

\br

(2) We may naturally extend $\overline{A}$ to a monotonic sequence of elements of $\mcp(I)$
indexed by $u \in \lao$, where $|u| \geq 2 \implies A_u = \emptyset$. 
Let $\langle A^\prime_u : u \in \lao$ be such a sequence. 
Apply Definition \ref{d:excellent} to obtain an excellent refinement $\langle B^\prime_u : u \in \lao \rangle$. Let $B_\alpha = B^\prime_{\{\alpha\}}$. 
Conditions (a)-(c) clearly hold. For condition (d), let $n \in \mathbb{N}, a_0, \dots a_{n-1} < \lambda$, $u = \{ a_0, \dots a_{n-1} \}$;
it suffices to check that the Boolean term 
\[ \sigma(x_{\{\alpha_0 \}}, \dots , x_{\{\alpha_{n-1} \}}) = x_{\{\alpha_0 \}} \cap \dots \cap x_{\{\alpha_{n-1} \}}  \in \Lambda_{\de, u} \]
If $w = u$, the demand is that $A^\prime_{\{\alpha_0\}} \cap \dots \cap A^\prime_{\{ \alpha_{n-1} \} } = \emptyset \mod \de$, which 
holds by hypothesis. If $w \subsetneq u$ then in the intersection we replace at least one $A_{\{\alpha_\ell\}}$ by $\emptyset$,
so the intersection is empty as desired. 
\end{proof}

\begin{rmk} \label{r:some-examples}  
Claim \ref{c:some-examples} remains true in the case where we replace $\mcp(I)$ by an arbitrary Boolean algebra $\ba$,
with the same proof.
\end{rmk}

For a characterization of excellence via goodness, see the Appendix p. \pageref{s:appendix}.

\section{\hspace{2mm}Separation of variables} \label{s:separation}

The main result of the section is ``separation of variables,'' Theorem \ref{t:separation}.  The reader may find it useful to refer to the discussion in
\S \ref{s:overview} above, which frames this result.

\br

\begin{defn} \label{d:poss} \emph{(Possibility patterns)}
Let $\ba$ be a Boolean algebra.
Say that $\overline{\ma}$ is a $(\lambda, \ba, T, \vp)$-possibility when:
\begin{enumerate}
\item $\overline{\ma} = \langle \ma_u : u \in \lao \rangle$ 
\item $u \in \lao$ implies $\ma_u \in \ba^+$
\item if $u \subseteq v \in \lao$ then $\ma_v \subseteq \ma_u$ (monotonicity)
\item if $u_* \in \lao$ and $\mb \in \ba^+$ satisfies
\[ \left( u \subseteq u_* \implies \left( ( \mb \leq \ma_u )  ~\lor~ ( \mb \leq 1 - \ma_u ) \right) \right) \land 
\left( \alpha \in u_* \implies \mb \leq \ma_{\{\alpha\}} \right) \]
then we can find a model $M \models T$ and $a_\alpha \in M$ for $\alpha \in u_*$ such that for every $u \subseteq u_*$,
\[ M \models (\exists x)\bigwedge_{\alpha \in u} \vp(x;a_\alpha) ~~ \mbox{iff} ~~ \mb \leq \ma_u \] 
\end{enumerate}
\end{defn}

\begin{disc}
In order to build ultrafilters on Boolean algebras which saturate a given theory, we will need a way to 
capture those sequences whose multiplicative refinements are truly necessary for, or visible to, the theory in question. 
A slogan for Definition \ref{d:poss} might be that ``any nonzero element of $\ba$ inducing an ultrafilter on $\overline{\ma}$ 
reveals a consistent $\vp$-configuration,'' e.g. in the sense of the characteristic sequences of \cite{mm2}.
\end{disc}

\begin{defn} \emph{(Moral ultrafilters on Boolean algebras)} \label{d:moral}
We say that an ultrafilter $\de$ on the Boolean algebra $\ba$ is $(\lambda, \ba, T, \vp)$-moral when
for every $(\lambda, \ba, T, \vp)$-possibility $\overline{a} = \langle a_u : u \in \lao \rangle$ satisfying
\begin{itemize} 
\item $u \in \lao \implies a_u \in \de$ 
\item $v \subseteq u \in \lao \implies a_u \subseteq a_v$
\end{itemize}
there is a multiplicative $\de$-refinement $\overline{b} = \langle b_u : u \in \lao \rangle$, i.e.
\begin{enumerate}
\item $u_1, u_2 \in \lao \implies b_{u_1} \cap b_{u_2} = b_{u_1 \cup u_2}$
\item $u \in \lao \implies b_u \subseteq a_u$
\item $u \in \lao \implies b_u \in \de$
\end{enumerate}
We write $(\lambda, \ba, T, \Delta)$-moral to mean $(\lambda, \ba, T, \vp)$-moral for all $\vp \in \Delta$, 
and $(\lambda, \ba, T)$-moral to mean for all formulas $\vp$ in the language of $T$, see Remark \ref{moral-rmk}.
\end{defn}

\begin{rmk} \label{moral-rmk}
Note that ``$(\lambda, \ba, T)$-moral'' in Definition \ref{d:moral} indeed means that morality holds locally, for each $\vp$. 
The global and local cases are not different in our context thanks to Fact \ref{phi-types}, or, in the case of a larger language,
Corollary \ref{phi-types-general}.
\end{rmk}

\begin{fact} \emph{(Local saturation implies saturation, \cite{mm1} Theorem 12)} \label{phi-types}
Suppose $\de$ is a regular ultrafilter on $I$ and $T$ a countable complete first order theory. Then for any $M^I/\de$, the following are equivalent:
\begin{enumerate}
\item $M^I/\de$ is $\lambda^+$-saturated.
\item $M^I/\de$ realizes all $\vp$-types over sets of size $\leq \lambda$, for all formulas $\vp$ in the language of $T$.
\end{enumerate}
\end{fact}

In this paper, we focus on countable theories for transparency, but we also record in Corollary \ref{phi-types-general} that the argument of 
\cite{mm1} \S 3 extends to larger languages; the only change is the limit stages, which follow from the stronger hypothesis that $\lcf(|T|, \de) \geq \lambda^+$.

\begin{cor} \label{phi-types-general}
Let $T$ be a complete first-order theory. Suppose $\de$ is a regular ultrafilter on $I$ which is $|T|^+$-good, or just
such that $\lcf(|T|, \de) \geq \lambda^+$. 
Then for any $M^I/\de$, the following are equivalent:
\begin{enumerate}
\item $M^I/\de$ is $\lambda^+$-saturated.
\item $M^I/\de$ realizes all $\vp$-types over sets of size $\leq \lambda$, for all formulas $\vp$ in the language of $T$.
\end{enumerate}
\end{cor}

\begin{defn} \emph{(Distributions)} \label{d:dist}
Let $T$ be a countable complete first-order theory, $M \models T$, $\de$ a regular ultrafilter on $I$, $|I| = \lambda$, $N = M^\lambda/\de$.
Let $p(x) = \{ \vp_i(x;a_i) : i < \lambda \}$ be a consistent partial type in the ultrapower $N$. Then a \emph{distribution} 
of $p$ is a map $d : \fss(\lambda) \rightarrow \mcp(I)$ which satisfies:
\begin{enumerate}
\item $\rn(d) \subseteq \de$
\item For each $\sigma \in [\lambda]^{<\aleph_0}$, 
$d(\sigma) \subseteq \{ t \in I : M \models \exists x \bigwedge \{ \vp_i(x;a_i[t]) : i \in \sigma \} \}$. Informally speaking,
$d$ refines the \Los map. 
\item $d$ is monotonic, meaning that for $\sigma, \tau \in [\lambda]^{<\aleph_0}$, $\sigma \subseteq \tau$ implies
$d(\sigma) \supseteq d(\tau)$
\item For each $t \in I$, $| \{ \sigma \in \lao : t \in d(\sigma) \} | < \aleph_0$. Note that in the presence of $(1)$, this implies
the range of $d$ is a regularizing family.
\end{enumerate}
A map satisfying $(2),(3),(4)$ is called a \emph{weak} distribution.
A distrubution satisfying the additional conditions of Obs. $\ref{dist-ac}$ is called \emph{accurate}. 
\end{defn}

\begin{conv}
We will often identify a distribution or weak distribution $d$ with the image of $\lao$ under $d$, i.e. with a sequence of the form 
$\langle A_u : u \in \lao \rangle \subseteq \mcp(I)$.
\end{conv}

\begin{obs} \label{dist-ac}
Let $M, N, T, I, p$ be as in Definition \ref{d:dist}. If $p$ has a (weak) distribution, we may choose a (weak) distribution $d$ of $p$ which is \emph{accurate}, 
where this means that in addition: for each $t \in I$ and $\sigma \subseteq \{ i < \lambda : t \in d(i) \}$, 
\[ M \models \exists x \bigwedge\{ \vp(x;a_i) : i \in \sigma \}  \iff t \in d(\sigma) \]
\end{obs}

\begin{proof}
Begin with a partial map $d: \lambda \rightarrow \de$ given by the \Los map, intersect with some regularizing family, and inductively extend $\de$ to all finite subsets of $\lambda$ 
to reflect the actual pattern of incidence in each index model. 
\end{proof}

\begin{lemma} \emph{(Transfer lemma)} \label{l:trans}
Suppose that we have the following data:
\begin{enumerate}
\item $\de$ is a regular, $\lambda^+$-excellent filter on $I$
%\item $\de_1$ is an ultrafilter on $I$ extending $\de$
\item $\ba$ is a Boolean algebra 
\item $\jj : \mcp(I) \rightarrow \ba$ is a surjective homomorphism with $\de = \jj^{-1}(\{ 1_\ba \})$
\item $\vp = \vp(x,y)$ is a formula of $T$.
%\item $\de_* = \{ \mb \in \ba : \jj^{-1}(\mb) \in \de_1 \}$
\end{enumerate}

Then the following two statements are true. 
\br
\begin{enumerate}
\item[(A)] Let $\langle A_u : u \in \lao \rangle \subseteq \mcp(I)$ be the image of an accurate weak distribution of some $\vp$-type. Then 
$\langle \jj(A_u) : u \in \lao \rangle \subseteq \ba$ is a $(\lambda, \ba, T, \vp)$-possibility pattern. 
\br
\item[(B)] Let $\langle \ma_u : u \in \lao \rangle$ be a $(\lambda, \ba, T, \vp)$-possibility pattern. Then there exists a sequence $\overline{A} = \langle A_u : u \in \lao \rangle \subseteq \mcp(I)$
such that $\jj(A_u) = \ma_u$ for each $u \in \lao$ and such that $\overline{A}$ is an accurate weak distribution of some $\vp$-type. 
\end{enumerate}
\end{lemma}

\br
\begin{proof} 
(A) Let $\overline{A} = \langle A_u : u \in \lao \rangle$ to be an accurate weak distribution of a given $\vp$-type $p$.
Let $\overline{\ma} = \langle \ma_u : u \in \lao \rangle$ be a sequence of elements of $\ba^+$ given by $\ma_u = \jj(A_u)$. 

We check that $\langle \ma_u : u \in \lao \rangle$ is a $(\lambda, \ba, T, \vp)$-possibility pattern. 
Conditions (a)-(c) of Definition \ref{d:poss} follow from the definitions of $\overline{A}$ and $\jj$. Recall that for condition
(d) we need to check that if $u_* \in \lao$ and $\mb \in \ba^+$ satisfies
\[ \left( u \subseteq u_* \implies \left( ( \mb \leq \ma_u )  ~\lor~ ( \mb\leq 1 - \ma_u ) \right) \right) \land 
\left( \alpha \in u_* \implies \mb\leq \ma_{\{\alpha\}} \right) \]
then we can find a model $M \models T$ and $a_i \in M$ for $i \in u_*$ such that for every $u \subseteq u_*$,
\[ M \models (\exists x)\bigwedge_{i \in u} \vp(x;a_i) ~~ \mbox{iff} ~~ \mb \leq \ma_u \] 
Let such $u_*$ and $\mb$ be given. Choose $B \in \mcp(I)$ such that $\jj(B) = \mb$. 
Then $B \neq \emptyset \mod \de$ since $\jj$ is a homomorphism. Moreover, 
if $\sigma, \tau$ partition $\mcp(u_*)$ such that
\[ \ba \models \bigwedge_{u \in \sigma} \mb \leq \ma_u \land \bigwedge_{v \in \tau} \mb \leq 1 - \ma_v \]
then we likewise have that 
\[   B \cap \left(  \bigcap_{u \in \sigma} A_u \setminus \bigcup_{v \in \tau} A_v   \right) \neq \emptyset \mod \de \]
Choose any $t$ belonging to this nonempty set. Then by accuracy of $\overline{A}$, $\{ a_i[t] : i \in u_* \}$ provide the desired witnesses.
Thus $\langle \ma_u : u \in \lao \rangle$ is a $(\lambda, \ba, T, \vp)$-possibility pattern, as desired.

\br
\br
(B) Let $\overline{\ma} = \langle \ma_u : u \in \lao \rangle$ be a $(\lambda, \ba, T, \vp)$-possibility.

First, for each $u \in \lao$ we may choose $A_u \subseteq I$ such that $\jj(A_u) = \ma_u$ by surjectivity of $\jj$. Apply
Definition \ref{d:excellent} to $\overline{A} = \langle A_u : u \in \lao \rangle$ to obtain an excellent refinement 
$\overline{B} = \langle B_u : u \in \lao \rangle$. Note that as $A_u = B_u \mod \de$, $\jj(A_u) = \jj(B_u)$ by definition of $\jj$. 
Now for each $t \in I$ let $\uu_t = \{ \epsilon < \lambda : t \in B_{\{\epsilon \}} \}$. Let us verify that 
for each $t \in I$ we can find $c_{t,\epsilon}$ in $M$ for $\epsilon \in \uu_t$ such that for every finite $u \subseteq \uu_t$,
we have 
\[ M \models \exists x \bigwedge \{ \vp(x;c_{t,\epsilon}) : \epsilon \in u \}  ~~ \iff ~~ t \in B_u \]
As $M$ is $\lambda^+$-saturated it suffices to prove this for fixed $t$ and $u_* = \{ \epsilon_0, \dots \epsilon_{n-1} \} \subseteq \uu_t$ finite.
Let $w_0, \dots w_k$ list the subsets $u$ of $u_*$ such that $t \in B_{u}$, and let $v_0, \dots v_m$ list the subsets
$u$ of $u_*$ such that $t \notin B_{u}$. We would like to find elements $c_0, \dots c_{n-1}$ of $M$ such that:
\[ M \models \bigwedge_{i \leq k} \left( \exists x \bigwedge_{\ell \in w_i} \vp(x;c_\ell)\right) 
\land \bigwedge_{j\leq m} \left( \neg \exists x \bigwedge_{\ell \in v_j} \vp(x,c_j) \right) \]

To do this we verify that we can always find $\mb \in \ba \setminus \{ 0 \}$ such that
$\mb \leq \bigcap_{i \leq k} \mb_{w_i}$ while $\mb \cap \left(\bigcup_{j \leq m} \mb_{v_j}\right) = 0$.
Suppose for a contradiction that there is no such $\mb$.
We have the corresponding Boolean term $\sigma(x_{\mcp(u_*)}) = \bigcap_i x_{u_i} \cap \bigcap_j 1 - x_{v_j}$.
Let us check whether $\sigma \in \Lambda_{\de, \overline{\mb}}$. If $u = u_*$, the term is $0_\ba$ by our assumption for a contradiction. 
If $u \subseteq u_*$ misses some $w_i$, then the expression is clearly $0_\ba$. 
But recall that the list of $w_i$ includes all singleton sets, by Definition \ref{d:poss}. So if
$u \subsetneq u_*$ we obtain $0_\ba$ as well. This shows that $\sigma \in \Lambda_{\de, \overline{\mb}}$ and thus
that $\sigma(\overline{B}|_{\mcp(u_*)}) = 0$ by the choice of $\overline{B}$. In other words, the supposed pattern does not occur for $u_*$
at any index $t$ for which $u_* \subseteq \uu_t$, contradiction. 
Thus we may always find $\mb \in \ba$ witnessing the pattern under consideration, and 
thus apply the definition of ``possibility'' \ref{d:poss} to choose parameters as desired.

Thus for each $t \in I$ we are able to choose $\{ c_{t,\epsilon} : \epsilon \in \uu_t \}$ as described. For $\epsilon \notin \uu_t$, 
let $c_{t, \epsilon}$ be arbitrary. Then for each $\epsilon < \lambda$ let $c_\epsilon = \prod_{t \in I} c_{t,\epsilon} /\de_1$. 
The type $p(x) = \{ \vp(x,c_\epsilon) : \epsilon < \lambda \}$ has accurate weak distribution $\langle B_u : u \in \lao \rangle$,
which completes the proof.
\end{proof}

\begin{theorem} \emph{(Separation of variables)} \label{t:separation}
Suppose that we have the following data:
\begin{enumerate}
\item $\de$ is a regular, $\lambda^+$-excellent filter on $I$
\item $\de_1$ is an ultrafilter on $I$ extending $\de$, and is ${|T|}^+$-good
\item $\ba$ is a Boolean algebra 
\item $\jj : \mcp(I) \rightarrow \ba$ is a surjective homomorphism with $\de = \jj^{-1}(\{ 1_\ba \})$
\item $\de_* = \{ \mb \in \ba : ~\mbox{if} ~\jj(A) = \mb ~\mbox{then}~ A \in \de_1 \}$
\end{enumerate}
Then the following are equivalent:
\begin{itemize}
\item[(A)] $\de_*$ is $(\lambda, \ba, T)$-moral, i.e. moral for each formula $\vp$ of $T$.
\item[(B)] For any $\lambda^+$-saturated model $M \models T$, $M^\lambda/\de_1$ is $\lambda^+$-saturated.
\end{itemize}
\end{theorem}

\begin{proof}
First we note that it suffices to replace the conclusion of (B) with ``$M^\lambda/\de_1$ is $\lambda^+$-saturated for $\vp$-types,
for all formulas $\vp$ of $T$,'' by Fact \ref{phi-types} (in the main case of interest for Keisler's order) 
or Corollary \ref{phi-types-general} (in general). Thus in what follows, we concentrate on $\vp$-types.
Note also that any regular ultrafilter is $\aleph_1$-good, thus (2) is always satisfied in the main case of interest, countable theories.

\br
$(A) \implies (B)$ 
Suppose that $\de_*$ is $(\lambda, \ba, T)$-moral, and we would like to show that $M^\lambda/\de_1$ is $\lambda^+$-saturated. 
Let $p = \{ \vp(x,a_i) : i < \lambda \}$ be the type in question, and let $\overline{A} = \langle A_u : u \in \lao \rangle$ be 
an accurate distribution of $p$, thus a sequence of elements of $\de_1$. It will suffice to show that $\overline{A}$ has a multiplicative refinement.

By (A) of the Transfer Lemma \ref{l:trans}, writing $\ma_u$ for $\jj(A_u)$, we have that
$\langle \ma_u : u \in \lao \rangle \subseteq \ba$ is a $(\lambda, \ba, T, \vp)$-possibility pattern. By 
hypothesis (5) each $\ma_u \in \de_*$. 

We had assumed that $\de_*$ is $(\lambda, \ba, T)$-moral, thus it contains a multiplicative refinement 
$\overline{\mb} = \langle \mb_u : u \in \lao \rangle$ of $\overline{\ma}$. 
Choose $\overline{B} = \langle B_u : u \in \lao \rangle$ so that $B_u \subseteq A_u$ and $\jj(B_u) = \mb_u$, 
for $u \in \lao$. Then $\overline{B}$ is a sequence of elements of $\de_1$ and is multiplicative modulo $\de$.
Applying excellence of $\de$, we may replace the sequence $\overline{B}$ with a 
$\lao$-indexed sequence $\overline{C}$ which refines $\overline{B}$, whose elements belong to $\de_1$, and which is truly multiplicative (by conditions
(1), (2), (3) of \ref{d:excellent}, respectively). A fortiori, $\overline{C}$ is a multiplicative refinement of $\overline{A}$,
which completes the proof. 

\br
\br
$(B) \implies (A)$
Suppose that $M^I/\de$ is $\lambda^+$-saturated and $\vp=\vp(x;y)$ is a formula in the language of $T$, recalling that
$\ell(y)$ need not be 1. 
We will show that $\de_*$ is $(\lambda, \ba, T, \vp)$-moral. 
Let $\overline{\ma} = \langle \ma_u : u \in \lao \rangle$ be a $(\lambda, \ba, T, \vp)$-possibility and we look for a multiplicative refinement. 

By (B) of the Transfer Lemma \ref{l:trans}
there exists $\overline{A} = \langle A_u : u \in \lao \rangle \subseteq \mcp(I)$
such that $\jj(A_u) = \ma_u$ for each $u \in \lao$ and such that $\overline{A}$ is an accurate weak distribution of some $\vp$-type $p$. 
By definition of $\de_1$, assumption (5) of the theorem, $\overline{A}$ is in fact an accurate distribution. Thus $p$ is a consistent
type in $M^I/\de$, therefore by our assumption $(B)$ it is realized. Let $\alpha$ be such a realization.  Then the sequence $\overline{\mb}$ defined by 
$\mb_u = \jj(\{ t \in I : M \models \vp(\alpha[t], c_{t, \epsilon}) \} )$ for $u \in \lao$ is the image of a multiplicative refinement of
$\overline{A}$, so will be a sequence of elements of $\de_*$ (thus of $\ba^+$) 
as well as a multiplicative refinement of $\overline{\ma}$. This completes the proof.
\end{proof}

Often we do not need to keep track of all formulas or patterns; some much smaller ``critical set'' will suffice for morality or saturation. 

\begin{defn} \emph{(Critical sets)} \label{d:critical}
Say that $\mathcal{C}_T =\{ \vp_i : i < i_* \leq |\mathcal{L}(T)| \}$ is a \emph{critical set of formulas for $T$}
if whenever $\lambda \geq \aleph_0$, $\de$ is a regular ultrafilter on $\lambda$ which is
$|T|^+$-good and $M \models T$, we have that $M^\lambda/\de$ is $\lambda^+$-saturated if and only if it is $\lambda^+$-saturated for
$\vp$-types for all $\vp \in \mathcal{C}_T$.
\end{defn}

\section{\hspace{2mm}Lemmas for existence} \label{s:l-e}

In this section we develop some machinery which will streamline the existence proof for excellent filters, Theorem \ref{t:existence}.
See Discussion \ref{d:e-m} below for an organizational discussion of the aims of this section, which we postpone until after several
definitions. 

We begin by recalling \emph{independent families of functions}, a useful tool for keeping track of the remaining decisions in
filter construction.

\begin{defn}
Given a filter $\de$ on $\lambda$, we say that a family $\eff$ of functions from $\lambda$ into $\lambda$ is
\emph{independent $\mod \de$} if for every $n<\omega$, distinct $f_0, \dots f_{n-1}$ from $\eff$ and choice of $j_\ell \in \rn(f_\ell)$,
\[ \{ \eta < \lambda ~:~ \mbox{for every $i < n, f_i(\eta) = j_i$} \} \neq \emptyset ~~~ \operatorname{mod} \de \]
\end{defn}

\begin{thm-lit} \label{thm-iff} \emph{(Engelking-Karlowicz \cite{ek} Theorem 3, see also Shelah \cite{Sh:c} Theorem A1.5 p. 656)}
For every $\lambda \geq \aleph_0$ there exists a family $\eff$ of size $2^\lambda$ with each $f \in \eff$ from 
$\lambda$ onto $\lambda$ such that $\eff$ is independent
modulo the empty filter \emph{(}alternately, by the filter generated by $\{ \lambda \}\emph{)}$. 
\end{thm-lit}

\begin{cor} \label{thm-2}
For every $\lambda \geq \aleph_0$ there exists a regular filter $\de$ on $\lambda$ and a family $\eff$ of size $2^\lambda$ which is independent
modulo $\de$.
\end{cor}

\begin{defn}
Let $\ba$ be a Boolean algebra. $CC(\ba)$ is the smallest regular cardinal $\lambda$ such that any maximal antichain of $\ba$ has cardinality
less than $\lambda$. If $\de$ is a filter on $I$, by $CC(B(\de))$ we will mean $CC(B)$ for $B = \mcp(I)/\de$.
\end{defn}

\begin{fact} \emph{(\cite{Sh:c} p. 359)} \label{f:cc}
Suppose $\de$ is a maximal filter on $I$ modulo which $\eff$ is independent. Then $CC(B(\de)) = \aleph_0$ iff for only finitely many
$f \in \eff$ is $|\rn(f)|>1$, and for no $f \in \eff$ is $|\rn(f)| = \aleph_0$. Otherwise $CC(B(\de))$ is the first regular cardinal
$\lambda > \aleph_0$ such that $f \in \eff$ implies $|\rn(f)| < \lambda$.  
\end{fact}

The following definition, ``good triple,'' is in current use, so we keep the name here and note that it does not imply that the filter is good in the 
sense of Definition \ref{d:good} (though this should not cause confusion). As usual, omitting ``pre-'' means being maximal for the given property.

\begin{defn} \emph{Good triples (cf. \cite{Sh:c} Chapter VI)} \label{good-triples}
\item The triple $(I, \de, \gee)$ is $(\lambda, \kappa)$-pre-good when:
\begin{enumerate}
\item $I$ is an infinite set of cardinality $\lambda$
\item $\de$ is a filter on $I$
\item $\gee$ is a family of functions from $I$ to $\kappa$
\item for each function $h$ from some finite $\gee_h \subseteq \gee$ to $\kappa$ such that $g \in \gee_h \implies h(g) \in \rn(g)$, 
we have that $A_h \neq \emptyset \mod \de$, where
\[ A_h = \{ t \in I ~:~ g \in G_h \implies g(t) = h(g) \} \]
\item $\fin(\gee) = \{ h ~:~\mbox{$h$ as just defined with $\operatorname{dom}(h)$ finite } \}$
\item $\fin_s(\gee) = \{ A_h : h \in \fin(\gee) \}$
\item We omit ``pre'' when $\de$ is maximal subject to these conditions.
\end{enumerate}
\end{defn}

\begin{obs} \label{good-dense} 
If $(I, \de, \gee)$ is a good triple, then $\fin_s(\gee)$ is dense in $\mcp(I) \mod \de$.
\end{obs}

\begin{proof}
We prove this in a more general case, Observation \ref{dense-new} below. 
\end{proof}

The next fact summarizes how such families allow us to construct ultrafilters; for more details,
see \cite{Sh:c} Chapter VI, Section 3.

\begin{fact} \label{uf} \emph{(\cite{Sh:c} Lemma 3.18 p. 360)}
Suppose that $\de$ is a maximal filter modulo which $\eff \cup \gee$ is independent, $\eff$ and $\gee$ are disjoint,
the range of each $f \in \eff \cup \gee$ is of cardinality less than $\operatorname{cof}(\alpha)$, $\operatorname{cof}(\alpha) > \aleph_0$,
$\eff = \bigcup_{\eta < \alpha} \eff_\eta$, the sequence $\langle \eff_\eta : \eta < \alpha \rangle$ 
is increasing, and let $\eff^\eta = \eff \setminus \eff_\eta$.
Suppose, moreover, that $D_\eta$ ($\eta<\alpha$) is an increasing sequence of filters which satisfy:

\begin{enumerate}
\item[(i)] Each $D_\eta$ is generated by $\de$ and sets supported $\mod \de$ by $\fin_s(\eff_\eta \cup \gee)$.
\item[(ii)] $\eff^\eta \cup \gee$ is independent modulo $D_\eta$. 
\item[(iii)] $D_\eta$ is maximal with respect to \emph{(i), (ii)}. 
\end{enumerate}
Then
\begin{enumerate}
\item $D^* := \bigcup_{\eta < \alpha} D_\eta$ is a maximal filter modulo which $\gee$ is independent.
\item If $\gee$ is empty, then $D^*$ is an ultrafilter, and for each $\eta < \alpha$, \emph{(ii)} is satisfied whenever $D_\eta$ is
non-trivial and satisfies \emph{(i)}.
\item If $\eta < \alpha$ and we are given $D^\prime_\eta$ satisfying \emph{(i), (ii)} we can extend it to a filter satisfying
\emph{(i), (ii), (iii)}. 
\item If $f \in \eff^\eta$ then $\langle f^{-1}(t) /D_\eta : t \in Range(f) \rangle$ is a partition in $B(D_\eta)$.  
\end{enumerate}
\end{fact}

\begin{defn} \label{d:ba} Denote by
$\ba_{\chi, \mu}$ the completion of the Boolean algebra generated by 
\\ $\{ x_{\alpha, \epsilon} : \alpha < \chi, \epsilon < \mu \}$ freely except for the conditions
$\alpha < \chi \land \epsilon < \zeta < \mu \implies x_{\alpha, \epsilon} \cap x_{\alpha, \zeta} = 0$.
\end{defn}

\begin{disc} \label{d:e-m} 
\emph{In our main construction, we first build an excellent filter whose quotient Boolean algebra admits a surjective homomorphism onto $\ba$, 
and then construct an ultrafilter on this $\ba$ using the method of independent functions.}

\emph{For the first stage, the set of tools developed beginning 
with Definition \ref{d:m1} will allow us to upgrade the notion of ``$\gee$ is independent $\mod \de$'' 
to take into account a background Boolean algebra $\ba$
which retains a specified amount of freedom. This will be used in the construction of Theorem \ref{t:existence}, 
where the intention will be that (by the end of the construction) we will have the desired map to $\ba$.}

\emph{For the second stage, 
Definition \ref{gb-triples} through Observation \ref{o:support}, which we now discuss, are direct translations of the
facts about families of independent functions, for the purposes of constructing ultrafilters on $\ba = \ba_{2^\lambda,\mu.}$}

\emph{Thus, in the case where $\ba$ from the first stage is taken to be $\ba_{2^\lambda, \mu}$, the Boolean algebra $\ba$ occurring throughout this section is essentially the same object, but its different uses correspond 
to the different stages in the proof.}
\end{disc}

\begin{defn} \emph{(Good Boolean triples)} \label{gb-triples}
\item The ``Boolean'' triple $(\ba, \de, \gee)$ is $(2^\lambda,\kappa)$-pre-good when:
\begin{enumerate}
\item $\ba = \ba_{2^\lambda, \kappa}$, so in particular $\ba$ has the $\kappa^+$-c.c.
\item $\de$ is a filter on $\ba$
\item $\gee = \{ \{ \mb_{i,\epsilon} : \epsilon < \mu \} : i < \kappa \}$ is a set of partitions of $\ba$
\item for each function $h$ from some finite $\sigma \subseteq \kappa$ to $\kappa$ 
we have that $A_h \neq \emptyset \mod \de$, where
\[ A_h = \bigcap \{ \mb_{i,\epsilon} ~:~ i \in \dm(h), h(i)=\epsilon \} \]
\item $\fin(\gee) = \{ h ~:~\mbox{$h$ as just defined with $\operatorname{dom}(h)$ finite } \}$
\item $\fin_s(\gee) = \{ A_h : h \in \fin(\gee) \}$
\item We omit ``pre'' when $\de$ is maximal subject to these conditions.
\end{enumerate}
\end{defn}

\begin{rmk}
Notice that according to our notation, ``Boolean'' triples are $(2^\lambda, ..., ...)$-good where
ordinary triples $(I, \de, \gee)$ would have been $(\lambda, ..., ...)$-good. This should not cause confusion.
Throughout the paper we consider independent families of size $2^\lambda$ and index sets of size $\lambda$. 
\end{rmk}

\begin{defn} \label{d:suitable}
Say that $\gee$ is a \emph{set of independent partitions of $\ba$ $\mod \de$} when
$(\ba, \de, \gee)$ is a pre-good Boolean triple as in Definition \ref{gb-triples}.
\end{defn}

\begin{obs} \label{o:support}
Let $(\ba, \de, \gee)$ be a $(2^\lambda, \mu)$-good Boolean triple. 
Let $\langle \mb_j : j < \lambda \rangle$ be a sequence of elements of $\ba$ which are each nonzero modulo $\de$. 
Then there exists a set $\gee^\prime \subseteq \gee$ of independent partitions, $|\gee^\prime| \leq \lambda$
such that for each $j<\lambda$ the element $\mb_j$ is supported in $\fin_s(\gee^\prime)$.

In particular, in the notation of Definition \ref{d:ba} this is true when 
$\gee$ is $\{ \{ \mx_{\alpha,\epsilon} : \alpha < \mu \} : i < 2^\lambda \}$ and
$\de = \{ 1_\ba \}$. 
\end{obs}

\begin{proof}
For each $j<\lambda$, choose a partition $\mpp_j = \{ A_{h^j_\ell} : \ell < \mu \}$ of $\fin_s(\gee)$ supporting $\mb_j$.
[How? We try to do this by induction on $\ell < \mu^+$ using the translation of Fact \ref{good-dense}. At odd steps
choose new elements for the partition in the ``remainder'' inside $\mb$, 
at even steps choose new elements in the ``remainder'' inside $1_\ba \setminus \mb$, and at limits take unions.
By the $\mu^+$-c.c. any such partition will stop at some bounded stage below $\mu^+$ as the remainders become empty, 
and then we may renumber so the partition is indexed by $\ell < \mu$.]

Let $X_j = \{ i < 2^\lambda : (\exists \ell < \mu)(i \in \dm(h^j_\ell) \}$ be the set of indices for ``rows'' in $\gee$ used
in the partition for $\mb_j$.
Let $\gee^\prime = \{ \mb_{i,\epsilon} : \epsilon < \mu, ~(\exists j < \lambda) ( i \in X_j ) \}$ collect all such ``rows.''
Then $|\gee^\prime| \leq \lambda$, and each $\mb_j$ is supported by $\fin_s(\gee^\prime)$ by construction, which completes the proof. 
\end{proof}

%[Excellence]

\br
This completes the introduction of notation for building ultrafilters on a specified Boolean algebra (the step corresponding to
``morality''). We now introduce notation for the complementary step, corresponding to ``excellence.''

%% In other words you constrain $\gee$ to look like $\ba$. 

\begin{defn} \emph{(Boolean algebra constraints on an independent family of functions)} \label{d:m1}
Fix $\lambda \geq \aleph_0$ and an index set $I$, $|I| = \lambda$, $\de_0$ a filter on $I$. 
Let $\ba$ be a complete Boolean algebra of cardinality $\leq 2^\lambda$ with $CC(\ba) \leq \lambda^+$, 
$\gee \subseteq {^\lambda}2$ a family of functions independent modulo $\de_0$, and $\de \supseteq \de_0$ a filter.
Fix in advance a choice of enumeration of $\langle \mb_\gamma : \gamma < \gamma_* \rangle$ of $\ba \setminus \{ 0_\ba \}$ 
and an enumeration $\langle g_\gamma : \gamma < \gamma_* \rangle$ of $\gee$.

Given this enumeration, which will remain fixed for the remainder of the argument, set
$B_\gamma = g^{-1}_\gamma\{1\}$ for $\gamma < \gamma_*$. 

Let $[\gee|\ba] = \{ X : X \subseteq I, I \setminus X \in \operatorname{Cond} \}$, 
where
\begin{align*}
\operatorname{Cond} = \{ \sigma(B_{\gamma_0}, \dots B_{\gamma_{n-1}}) ~~:~~ & \sigma(x_0, \dots x_{n-1}) ~\mbox{is a Boolean term and} \\
& \gamma_0, \dots, \gamma_{n-1} < \gamma_* = |\ba| ~\mbox{are such that} \\
& \ba \models \mbox{``$ \sigma(\mb_{\gamma_0}, \dots \mb_{\gamma_{n-1}}) = 0$''} \}
\end{align*} 

\noindent Say that \emph{$\gee$ is constrained by $\ba$ modulo $\de$} when, for some choice of enumeration of $\ba$ and of $\gee$ inducing a 
definition of $\{ B_\gamma : \gamma < \gamma_* \}$, we have that $\de$ contains the filter generated by $[\gee|\ba]$.
\end{defn}

\begin{rmk} Remarks on Definition $\ref{d:m1}$:
First, we could have taken $\gee$ to be any family of functions with range of size $\geq 2$, e.g. $\mu$.
Second, this Definition looks towards Observation \ref{o:cong} and Observation \ref{o:works} below. 
\end{rmk}

In the main case of interest, the constraints on $\gee$ given by the Boolean algebra $\ba$ are the only barriers to independence:

\begin{defn} \label{d:ba-filter1}
In the notation of Defintion \ref{d:m1}, let $\de$ be a filter on $I$, $|I| = \lambda$.

\br
\begin{enumerate}
\item Suppose that $\gee$ is constrained by $\ba$ modulo $\de$. Fix enumerations of $\gee, \ba$ witnessing this. 

\item Let $\fin(\gee) = \{ h ~: ~\dm(h) \subseteq \gee, |\dm(h)| < \aleph_0, g \in \dm(h) \implies h(g) \in \rn(g) \}$

\item Say that $h \in \fin(\gee)$ is \emph{prevented by $\ba$} when  
\\ $h = \{ (g_\gamma, i_\gamma) : \gamma \in \sigma \in [\gamma_*]^{<\aleph_0},
i_\gamma \in \{ 0, 1\} = \rn(g_\gamma) \}$ and we have that
\[ \ba \models ~\mbox{``} \bigcap_{\gamma \in \sigma} (\mb_{\gamma})^{i_\gamma} ~ = 0 \mbox{''} \]
recalling Convention \ref{c:1} on exponentiation.

\item Now define
\begin{enumerate}
\item $\fin(\gee |\ba ) = \{ h ~: h \in \fin(\gee) ~\mbox{and $h$ is not prevented by $\ba$} \}$ 
\item $\fin_s(\gee|\ba) = \{ A_h : h \in \fin(\gee|\ba) \}$ 
\end{enumerate}

\br
\item Say that $(I, \de, \gee|\ba)$ is a $(\lambda, \mu)$-pre-good triple when for some enumeration of $\gee$ and $\ba$:
\begin{enumerate}
\item $I$ is an infinite set of cardinality $\lambda$
\item $\de$ is a regular filter on $I$
\item $\gee$ is a family of functions from $I$ to $2$
\item $\ba$ is a $\mu^+$-c.c. complete Boolean algebra, $|\ba| \leq 2^\lambda$
\item $\gee$ is constrained by $\ba$ modulo $\de$
\item for each $A_h \in \fin_s(\gee|\ba)$, $A_h \neq \emptyset \mod \de$
\item We omit ``pre'' when $\de$ is maximal subject to these conditions.
\end{enumerate}

\br
\item Say that $(I, \de, (\gee|\ba) \cup \eff)$ is a $(\lambda, \kappa, \mu)$-pre-good triple when:
\begin{enumerate}
\item $\eff \subseteq {^I\kappa}$, $\gee \subseteq {^I2}$, $\eff \cap \gee = \emptyset$
\item $(I, \de, \gee|\ba)$ is a $(\lambda, \mu)$-pre-good triple, so in particular $\ba$ is a $\mu^+$-c.c. Boolean algebra
\item for each $A_h \in \fin_s(\gee|\ba)$ and each $A_j \in \fin_s(\eff)$, 
$A_h \cap A_j \neq \emptyset \mod \de$.
\end{enumerate}

\br
$($N.b. $(\lambda, \kappa, \mu)$ means $\lambda = |I|$, $\kappa$ is the range of $f \in \eff$, and $\ba$ has the $\mu^+$-c.c.$)$ 
\end{enumerate}

\br
\item In the current paper, we focus on the case where $\mu < \lambda \leq 2^\mu$, $\ba$ has the $\mu^+$-c.c.,
and $\eff \subseteq {^I\lambda}$. To simplify notation, we will write ``$(\lambda, \mu)$-pre-good triple'' or
``$(\lambda, \mu)$-good triple'' for this case, or simply ``pre-good'' and ``good'' when the cardinal constraints are clear from context. 
\end{defn}

The next observation verifies that constraint by a Boolean algebra still yields a filter. 

\begin{obs} \label{o:works}
Let $\lambda, I, \de, \gee, \ba$ be as in Definition \ref{d:ba-filter1}. 
Suppose $\eff$ is a family of functions from $\lambda$ onto $\mu$, $\mu \leq \lambda$,
such that $\eff \cap \gee = \emptyset$ and $(I, \de, \gee \cup \eff)$ 
is a pre-good triple, so in particular $\gee \cup \eff$ is independent $\mod \de$. 

Then letting $\de_1$ be the filter generated by $\de \cup [\gee|\ba]$, in the notation of Definition \ref{d:m1}, we have that:
\begin{enumerate}
\item[(I)] $\de_1$ is a filter on $I$, and 
\item[(II)] for every $h \in \fin(\eff)$ and $\gamma < \gamma_*$, $A_h \cap B_\gamma \neq \emptyset \mod \de_1$. 
\end{enumerate}
\end{obs}

\begin{proof}
Since $\operatorname{Cond}$ is closed under finite disjunction it suffices to
show that any one of its elements $\sigma(\overline{B})$ has $\de_0$-nontrivial complement. To see this, put the negation of the corresponding Boolean term,
$\neg \sigma(\overline{x})$, in disjunctive normal form. Since $\ba \models$ ``$\neg \sigma(\mb_{\gamma_0}, \dots \mb_{\gamma_{n-1}}) = 1_\ba$'', we can choose
a disjunct $\tau$ which is nonzero in $\ba$. Then $\tau(\overline{B})$ will be a conjunction of literals, from which we can inductively construct
$A_h \in \fin(\gee)$ such that $A_h \subseteq \tau(\overline{B}) \mod \de_0$, simply by replacing each literal of the form ``$B_{\gamma_i}$''
by the condition $g_{\gamma_i} = 1$ and each literal of the form ``$\neg B_{\gamma_j}$'' by the condition $g_{\gamma_j} = 0$. Recall that
the range of each $g \in \gee$ is $\{0,1\}$. By choice of $\tau$,
this $h$ is indeed consistent so $A_h \neq \emptyset \mod \de_0$ since the family $\gee$ is independent.
Moreover $A_h$ is contained in $\lambda \setminus \sigma(\overline{B})$ by construction.
This shows that the complement of any set in $\operatorname{Cond}$ contained some element of $\fins(\gee)$ modulo $\de_0$, which completes the proof
of (I)-(II).
\end{proof}

For completeness we spell out that such objects exist. 

\begin{cor} \label{f:basic}
Let $\kappa \leq \lambda$ and let $\ba$ be a $\kappa^+$-c.c. complete Boolean algebra. Then there exist $\de, \eff, \gee$ such that:
\begin{enumerate}
\item $\de$ is a regular filter on $I$, $|I| = \lambda$
\item $\eff$ is a family of functions from $\lambda$ into $\kappa$, $|\eff| = 2^\lambda$
\item $\gee$ is a family of functions from $\lambda$ into $2$, $|\gee| = |\ba|$
\item $\eff \cap \gee = \emptyset$
\end{enumerate}
and $(I, \de, (\gee|\ba) \cup \eff)$ is a $(\lambda, \kappa)$-good triple.
\end{cor}

\begin{proof}
Let $\eff_0$ be the independent family given by Corollary \ref{thm-2} above. Without loss of generality we can write $\eff_0$ as the disjoint 
union of $\eff$ and $\gee$ satisfying (2)-(3). As each $A_h \in \fin_s(\eff \cup \gee)$ has cardinality $\lambda$, 
$(I, \de_0, \eff \cup \gee)$ remains pre-good for $\de_0 = \{ A \subseteq \lambda : | \lambda \setminus A | < \lambda \}$. 

Let $\gamma_* = |\ba|$.
Let $\langle f_\alpha : \alpha < 2^\lambda \rangle$, $\langle g_\gamma : \gamma < \gamma_* \rangle$, $\langle \mb_\gamma : \gamma < \gamma_* \rangle$
list $\eff$, $\gee$, and $\ba \setminus \{ 0_\ba \}$ respectively.
Let $B_\gamma = g^{-1}_\gamma\{1\}$ for $\gamma < \gamma_*$, and as usual define the set of conditions:
\begin{align*}
\operatorname{Cond} = \{ \sigma(B_{\gamma_0}, \dots B_{\gamma_{n-1}}) ~~:~~ & \sigma(x_0, \dots x_{n-1}) ~\mbox{is a Boolean term and} \\
& \gamma_0, \dots, \gamma_{n-1} < \gamma_* = |\ba| ~\mbox{are such that} \\
& \ba \models \mbox{``$ \sigma(\mb_{\gamma_0}, \dots \mb_{\gamma_{n-1}}) = 0$''} \}
\end{align*}
Let $\de_1$ be the filter generated by $\mathcal{A} = \{ X : X \subseteq \lambda, \lambda \setminus X \in \operatorname{Cond} \} \cup \de_0$.
Then by Observation \ref{o:works} $\de_1$ is a filter on $\lambda$ and $(I, \de, (\gee|\ba) \cup \eff)$ is $(\lambda, \kappa)$-pre-good.
To finish, let $\de \supseteq \de_1$ be maximal subject to the constraint that $(I, \de, (\gee|\ba) \cup \eff)$ is $(\lambda, \kappa)$-pre-good.
\end{proof}

\begin{obs} \label{o:cong}
If $(I, \de, \gee|\ba)$ is a good triple, then there is a 
surjective homomorphism  $\jj : \mcp(I) \rightarrow \ba = \ba_{2^\lambda, \mu}$ such that $\de = \jj^{-1}(\{ 1_\ba \})$.  
\end{obs}

\begin{obs} \label{dense-new} 
If $(I, \de, (\gee|\ba) \cup \eff)$ is a good triple, then $\fin_s((\gee|\ba) \cup \eff)$ is dense in $\mcp(I) \mod \de$.
\end{obs}

\begin{proof} (Just as in the usual proof, \cite{Sh:c} VI.3)
Recall that the definition of ``good triple'' assumes that $\de$ is maximal so that 
$(\gee|\ba) \cup \eff$ is independent $\mod \de$. 
Suppose for a contradiction the statement of the claim fails, i.e. 
that there is some $X \subseteq I$, $X \neq \emptyset \mod \de$ but for every 
$A_h \in \fin_s((\gee|\ba) \cup \eff)$, $A_h \not\subseteq X \mod \de$. Thus for each such $A_h$, 
$A_h \cap (I \setminus X) \neq \emptyset \mod \de$. Let $\de^\prime$ be the filter generated by 
$\de \cup \{ I \setminus X \}$. Then $(I, \de^\prime, (\gee|\ba) \cup \eff)$ is also a good triple,
contradicting the assumption about the maximality of $\de$.
\end{proof}

\br
With this notation in place, we now give two proofs. 
Recall that in proving existence of excellent filters, we will want on the one hand to ensure the quotient of $\mcp(I)$ modulo the final filter is isomorphic to a given complete Boolean algebra $\ba$, and on the other to ensure existence of excellent refinements. The final two results
of the section will form the corresponding inductive steps.

Roughly speaking, the following lemma says that if we have a filter $\de$ on $I$, a family $\eff$ with range $\lambda$, 
a family $\gee$ constrained by $\ba$ and a subset $X$ of the index set, we may extend $\de$ to 
a filter $\de^\prime$ so that $X$ is equivalent modulo $\de^\prime$ to some element of $\ba$ (by condition (d) and completeness) at the cost
of $\leq \lambda$ elements of $\eff$.

\begin{lemma} \label{l:free}
Suppose $(I, \de, (\gee|\ba) \cup \eff)$ is a $(\lambda, \lambda, \mu)$-good triple. Let $X \subseteq I$. Then there are 
$\de^\prime \supseteq \de$ and $\eff^\prime \subseteq \eff$ such that 
\begin{enumerate}
\item[(a)] $\de^\prime$ is a filter extending $\de$
\item[(b)] $|\eff| = 2^\lambda$, $|\eff \setminus \eff^\prime| \leq \lambda$
\item[(c)] $(I, \de^\prime, (\gee|\ba) \cup \eff^\prime)$ is a $(\lambda, \lambda, \mu)$-good triple
\item[(d)] $X$ is supported by $\fin_s(\gee|\ba)$ modulo $\de^\prime$
\end{enumerate}
\end{lemma}

\begin{proof}
Our strategy is as follows. 
By inductively consuming functions from $\eff$, we build a partition of $I$ using elements from $\fin_s(\gee|\ba)$ which 
are either entirely inside or entirely outside the fixed set $X$. 
The internal approximation at stage $\alpha$ we call $J^1_\alpha$, and the external approximation we call $J^0_\alpha$.
Since $X$ need not be supported by $\fin_s(\gee|\ba) \mod \de$, we must continually consume functions from 
$\eff$ in order to ``clarify the picture'' at stage $\alpha$ in a larger filter $\de_\alpha$, where we can continue to construct the partition.
We consume finitely many functions from $\eff$ at each successor stage, and at limits take unions. 
We apply Fact \ref{f:cc} to show the construction will stop before $\lambda^+$. 
When the induction stops, we have our desired partition (d), and this will complete the proof. 

More formally, we try to choose by induction on $\alpha < \lambda^+$ objects $J^\alpha_0, J^\alpha_1, \eff^\alpha, \de_\alpha$ to satisfy:
\begin{enumerate}
\item $J^\alpha_1 \subseteq \fin_s(\gee|\ba)$, and $\bigcup J^\alpha_1 \subseteq X \mod \de_\alpha$
\item $A, A^\prime \in J^\alpha_1 \implies A \cap A^\prime = \emptyset \mod \de_\alpha$
\item $\beta < \alpha \implies J^\beta_1 \subseteq J^\alpha_1$
\item $J^\alpha_0 \subseteq \fin_s(\gee|\ba)$, and $\bigcup J^\alpha_1 \subseteq I \setminus X \mod \de_\alpha$
\item $A, A^\prime \in J^\alpha_1 \implies A \cap A^\prime = \emptyset \mod \de_\alpha$
\item $\beta < \alpha \implies J^\beta_0 \subseteq J^\alpha_0$
\item $\de_\alpha$ is a filter extending $\de$, with $(I, \de_\alpha, (\gee|\ba) \cup \eff^\alpha)$ a good triple
\item $\beta < \alpha  \implies \de_\beta \subseteq \de_\alpha$
\item for $\alpha$ limit, $\de_\alpha = \bigcup \{ \de_\beta : \beta < \alpha \}$
\item $\eff^\alpha \subseteq \eff$, with $\beta < \alpha  \implies \eff^\alpha \subseteq \eff^\beta$
\item for $\alpha$ limit, $\eff^\alpha = \bigcap \{ \eff^\beta : \beta < \alpha \}$
\item for $\alpha = \beta + 1$, $|\eff^\beta \setminus \eff^\alpha| < \aleph_0$
\end{enumerate}

For $\alpha = 0$, let $\de_0 = \de$, $\eff_0 = \eff$. Choose $J^0_1, J^0_0$ maximal subject to the constraints
(1)-(4). 

For $\alpha$ limit, define $\de_\alpha$ in accordance with (9), $\eff^\alpha$ in accordance with (11). Likewise, let
$J^\alpha_1 = \bigcup \{ J^\beta_1 : \beta < \alpha \}$ and let 
$J^\alpha_0 = \bigcup \{ J^\beta_0 : \beta < \alpha \}$. 

For successor stages, we distinguish between even (increase $J^\beta_0$) and odd (increase $J^\beta_1$).

First consider $\alpha = \beta+1$. If both $X \setminus J^\beta_1 = \emptyset \mod \de_\beta$ and
$(I \setminus X) \setminus J^\beta_0 = \emptyset \mod \de_\beta$, then we satisfy (d) and finish. If 
$X \setminus J^\beta_1 = \emptyset \mod \de_\beta$, we have finished the construction of $J_1 = J^\beta_1$. 
So suppose $X \setminus J^\beta_1 = \emptyset \mod \de_\beta$. 

Apply Claim \ref{dense-new} to find sets $A_h \in \fin_s(\gee|\ba)$ and $A_{h^\prime} \in \fin_s(\eff)$ such that 
\[ \left( A_h \cap A_i \right) \subseteq \left( X \setminus J^\beta_1 \right) \mod \de \]
Keeping in mind the asymmetry between the roles of $\eff$ and $\gee$ in this proof,
let $\eff^\alpha = \eff^\beta \setminus \{ \dm h^\prime \}$, which satisfies (12) by definition of $\fin_s(\eff)$.
Now let $\de_\alpha$ be the filter generated by $\de_\beta \cup \{ A_{h^\prime} \}$. Clearly this is a filter by condition (7)$_\beta$,
and by construction, $(I, \de_\alpha, (\gee|\ba) \cup F^\alpha)$ will satisfy (7)$_\alpha$.
Finally, let $J^\alpha_1 = J^\beta_1 \cup \{ A_h \}$ and let $J^\alpha_0 = J^\beta_0$. 

The case $\alpha = \beta + 2$ is parallel to the odd case but with $J^{\beta+1}_0, J^\alpha_0$ replacing
$J^\beta_1, J^\alpha_1$, and $I \setminus X$ replacing $X$.

Without loss of generality, let each $\de_\alpha$ be maximal subject to the fact that 
$(I, \de_\alpha, (\gee|\ba) \cup \eff^\alpha)$ is a pre-good triple.

Note that for each $\alpha \leq \lambda^+$, $CC(B(\de_\alpha)) \leq \lambda^+$ by condition (7) and Fact \ref{f:cc}, and moreover
$J^\alpha_0 \cup J^\alpha_1$ is a set of pairwise disjoint elements of $B(\de_\alpha)$ with
$|J^\alpha_0 \cup J^\alpha_1| \geq |\alpha|$.  
Thus the length of this construction is bounded below $\lambda^+$. 
In other words, at some point $\alpha < \lambda^+$
both $X \setminus J^\beta_1 = \emptyset \mod \de_\alpha$ and $(I \setminus X) \setminus J^\beta_0 = \emptyset \mod \de_\alpha$,
and here we take $\de^\prime = \de_\alpha$ to finish the proof. 
\end{proof}

To conclude this section, we show how to construct an excellent refinement. 

\begin{claim} \label{e-ind-step}
Suppose $(I, \de, (\gee|\ba) \cup \eff)$ is $(\lambda, \lambda, \mu)$-good. 
Let $f \in \eff$, $\eff^\prime = \eff \setminus \{ f \}$, and let $\overline{A} = \langle A_u : u \in \lao \rangle$ 
be a sequence of elements of $\mcp(I)$ 
such that for each $A_u \in \overline{A}$ and each $A_h \in \fin_s(\eff)$, $A_h \cap A_u \neq \emptyset \mod \de$.

Then there is a filter $\de^\prime \supseteq \de$ and a sequence
$\langle B_u : u \in \lao \rangle$ of elements of $\mcp(I)$ satisfying Definition \ref{d:excellent}, namely:
\begin{enumerate}
\item for each $u \in \lao$, $B_u \subseteq A_u$
\item for each $u \in \lao$, $B_u = A_u \mod \de$
\item \emph{if} $u \in \lao$ and $\sigma \in \Lambda = \Lambda_{\de, \overline{A}|_u}$, so $\sigma(\overline{A}|_{\mcp(u)}) = \emptyset \mod \de$, \emph{then} $\sigma(\overline{B}|_{\mcp(u)}) = \emptyset$
\end{enumerate}
such that, moreover, $(I, \de^\prime, (\gee|\ba) \cup \eff^\prime)$ is $(\lambda, \lambda, \mu)$-good. 
\end{claim}

\begin{proof} We proceed in stages.

\step{Step 1: The exceptional sets $Y_\epsilon$.}
Let $\langle u_\epsilon : \epsilon < \lambda \rangle$ enumerate the finite
subsets of $\lambda$. 
For each $\epsilon$ let 
\[ Y_\epsilon = \bigcup \{ \sigma(\overline{A}|_{\mcp(u_\epsilon)}) : ~\mbox{$\sigma$ a Boolean term and 
$\sigma(\overline{A}|_{\mcp(u_\epsilon)}) = \emptyset \mod \de$} \}\]
Then $Y_\epsilon$ is a finite union of subsets of $\lambda$ which are $= \emptyset \mod \de$, hence 
$Y_\epsilon = \emptyset \mod \de$. 

\br
\step{Step 2: The role of ``below''.}
Suppose that  $\sigma(\overline{x}_{\mcp(u_\epsilon)}) \in  \Lambda_{\de, \overline{A}}$ and $v \subseteq u_\epsilon$. 
Let $\overline{A}^\prime$ be the sequence given by
$A^\prime_w = A_w$ if $w \subseteq v$ and $A^\prime_v = 0_\ba$ otherwise.
Then also $\sigma(\overline{A}^\prime|_{\mcp(u_\epsilon)}) \subseteq Y_{\epsilon}$, by definition of $\Lambda$ (as $0$ is a constant of the language of
Boolean algebras).

\br
\step{Step 3. Defining the filter.}
Now for $u \in \lao$ define
\[ B_u = \bigcup \{ A_u \cap f^{-1} (\epsilon) \setminus Y_\epsilon ~:~ \mbox{$\epsilon < \lambda$ and $u \subseteq u_\epsilon$} \} \]
Then $\langle B_u : u \in \lao \rangle$ is the proposed refinement. 
Let $\de^\prime$ be the filter generated by 
$\de \cup \{ X_u : u \in \lao \}$ where 
\[ X_u = I \setminus \left( B_u \Delta A_{u} \right) \]

\br
\step{Step 4. The filter is nontrivial and the triple is pre-good.}
Fix $u = u_\epsilon$. Then $f^{-1}(\epsilon) = f^{-1}(\epsilon)\setminus Y_\epsilon \mod \de$, so
\[ B_u \supseteq A_u \cap f^{-1}(\epsilon) \setminus Y_\epsilon  \neq \emptyset \mod \de \]
where ``$\neq \emptyset$'' holds by the hypothesis of independence, since $f \in \eff$. 
Likewise we assumed that any $A_h \in \fin_s(\eff^\prime)$,
$A_u \cap A_h \neq \emptyset \mod \de$. Recall $\eff^\prime = \eff \setminus \{ f \}$. Since $f \notin \eff^\prime$, we therefore have 
\[ B_u \supseteq A_u \cap A_h \cap (f^{-1}(\epsilon) \setminus Y_\epsilon) \neq \emptyset \mod \de \]
\noindent Thus $\de^\prime$ is a filter. Now consider $A_{h^\prime} \in \fin_s(\gee|\ba)$. There are two cases.
If $A_u \cap A_h \cap A_{h^\prime} \neq \emptyset \mod \de$, then this intersection is contained in $B_u$
$\mod \de$. Otherwise, $A_h \cap A_{h^\prime} \subseteq (I \setminus A_u) \mod \de$. In either case,
$A_h \cap A_{h^\prime}$ does not nontrivially intersect $(I \setminus (A_u \Delta B_u)) \mod \de$, 
and so remains nonempty $\mod \de^\prime$.
This shows that $(I, \de^\prime, (\gee|\ba) \cup \eff^\prime)$ is pre-good triple.

\br
\step{Step 5. The sequence $\overline{B}$ is excellent.} Write $\overline{B} = \langle B_u : u \in \lao \rangle$.
We have shown (1)-(2) from the statement of the claim.
For condition (3), let $u \in \lao$ and $\sigma(\overline{x}_{\mcp(u)}) \in \Lambda$ be given. 
Suppose for a contradiction that there were $t \in \sigma(\overline{B}|_{\mcp(u)})$.  Let $\epsilon_t =  f^{-1}(t)$. 
Then Step 2 in the case $v = u \cap u_{\epsilon_t}$ gives a contradiction.

\br
\step{Step 6: A good triple.} To finish, without loss of generality we may take $\de^\prime$ maximal subject to the condition that
$(I, \de^\prime, (\gee|\ba) \cup \eff^\prime)$ is a pre-good triple. 
\end{proof}

\section{\hspace{2mm}Existence} \label{s:existence}

In this section we prove that for any complete $\lambda^+$-c.c. Boolean algebra $\ba$ of cardinality $\leq 2^\lambda$ there is a regular
$\lambda^+$-excellent filter $\de$ on $\lambda$ such that $\ba$ is isomorphic to $\mcp(\lambda)/\de$.

\begin{theorem} \emph{(Existence)} \label{t:existence}
Let $\mu \leq \lambda$ and let $\ba$ be a $\mu^+$-c.c. complete Boolean algebra of cardinality $\leq 2^\lambda$. 
Then there exists a regular excellent filter 
$\de$ on $\lambda$ and a surjective homomorphism  $\jj : \mcp(I) \rightarrow \ba = \ba_{2^\lambda, \mu}$ such that $\de = \jj^{-1}(\{ 1_\ba \})$.  
\end{theorem}

\begin{proof} We give the proof in several stages. Recall that we have identified the index set with $\lambda$.

\step{Stage 0: Preliminaries.} 
We begin by choosing $\gee$, $\eff$, $\de_1$ such that $|\gee| = |\ba|$, $|\eff| = 2^\lambda$, $\gee$ is a family of functions
from $I$ onto $2$, $\eff$ is a family of functions from $\lambda$ onto $\lambda$, and 
$(I, \de_1, (\gee|\ba) \cup \eff)$ is a $(\lambda, \lambda, \mu)$-good Boolean triple, in the notation of
Definition \ref{d:ba-filter1}.  Such triples exist by Corollary \ref{f:basic}.

\step{Stage 1: Setting up the inductive construction.}
We now set up the construction of $\de$, which we build by induction on $\alpha < 2^\lambda$. Recall that we will want to ensure that on the one hand,
the final filter $\de$ is excellent, and on the other that the quotient $\mcp(I)/\de$ is exactly $\ba$. 

We address the issue of the quotient by enumerating $\mcp(\lambda)$ as $\langle C_\alpha : \alpha < 2^\lambda \rangle$ and 
ensuring, at odd inductive steps, that the set $C_\beta$ under consideration has an appropriate image. This suffices by
Observation \ref{o:cong}.

In order to address all possible barriers to excellence, at even inductive steps, 
we will need an enumeration of all sequences $\overline{B}$ as in Definition \ref{d:excellent}.
Say that $\mx : \fss(\lambda) \rightarrow \fss(\gamma_*)$ is an
\emph{indexing sequence} whenever 
\[ u \in \lao \implies \bigcup \{ \mx(i) : i \in u \} = \mx(u) \]
Let $\langle \mx_\alpha : \alpha < 2^\lambda \rangle$ list all indexing sequences, each appearing $2^\lambda$ times.
Below, given $u \in \lao$, we will write e.g. ``$\overline{\ma}_{\mx(\mcp(u))}$'' to indicate the finite sequence of elements of $\ba$ 
indexed by the image of the finite subsets of $u$ under $\mx$.

Now we choose $\de_{2,\alpha}$, $\eff^\alpha$ by induction on $\alpha \leq 2^\lambda$ such that:

\begin{enumerate}
\item $\de_{2,\alpha}$ is a filter on $\lambda$
\item $\beta < \alpha < 2^\lambda \implies \de_{2,\beta} \subseteq \de_{2, \alpha}$, and $\alpha$ limit 
implies $\de_{2,\alpha} = \bigcup_{\beta < \alpha} \de_{2,\beta}$
\item $\eff^\alpha \subseteq \eff$, $|\eff^{\alpha}| = 2^\lambda$, and $\beta < \alpha \implies \eff^\beta \supseteq \eff^\alpha$
\item If $h \in \fin(\eff^\alpha)$ and $\gamma < \gamma_*$ then $A_h \cap B_\gamma \neq \emptyset \mod \de_{2,\alpha}$
\item $(I, \de_\alpha, (\gee|\ba) \cup \eff^\alpha)$ is $(\lambda, \lambda, \mu)$-good
\item If $\alpha = 2\beta + 1$, then for some $\gamma < \gamma_*$, $C_\beta = B_\gamma$ $\mod \de_{2,\alpha}$.
\item If $\alpha = 2\beta + 2$, \emph{if} for each $u \in \lao$ and Boolean term $\sigma= \sigma(\xpu)$ we have that
\[ \ba \models \mbox{``}\sigma(\overline{\mb}_{\mx(\mcp(u))}) = 0\mbox{''} \implies \sigma(\overline{B}_{\mx(\mcp(u))}) = \emptyset \mod \de_{2,2\beta+1} \]
\emph{then} we can find $\overline{B}_\alpha = \langle B^\alpha_u : u \in \lao \rangle$ satisfying Definition \ref{d:excellent}.
\end{enumerate}

For $\alpha=0$ this is trivial: let $\de_{2,\alpha} = \de_1$, $\eff^\alpha = \eff$.

For $\alpha$ limit let $\de_{2,\alpha} = \bigcup \{ \de_{2,\beta} : \beta < \alpha \}$, $\eff^\alpha = \bigcap \{ \eff^\beta : \beta < \alpha \}$.

For $\alpha$ successor, we distinguish between even and odd.
\step{Stage 2: Odd successor steps.} For $\alpha = 2\beta + 1$ we address (5) for the given $C_\beta$. If $C_\beta = \emptyset \mod \de_{2,2\beta}$,
let $\de_{2,\alpha} = \de_{2,2\beta}$ and finish. Otherwise, apply Lemma \ref{l:free} above in the case where $\de = \de_{2,2\beta}$, 
$\eff = \eff^{2\beta}$, $X = C_\beta$. Then let $\de_{2,2\beta+1}$ be the filter $\de^\prime$ and let
$\eff^{2\beta + 1}$ be the family $\eff^\prime$ returned by that Lemma. 
Without loss of generality, let $\de_{2,2\beta+1}$ be maximal subject to the condition that $(\gee|\ba) \cup \eff$ remain independent. 
Note that conditions (3),(4),(5) are guaranteed by the statement of Lemma \ref{l:free}. 

\step{Stage 3: Even successor steps.} For $\alpha = 2\beta + 2$ we address condition (6). Suppose then that we are given an indexing
function $\mx=\mx_\alpha$ and a corresponding sequence $\langle B_{\mx(u)} : u \in \lao \rangle$ 
of elements of $\mcp(I)$. If the ``if'' clause in condition (4) fails,
let $\de_{2,\alpha} = \de_{2,2\beta+1}$, and see the bookkeeping remark in the next step. 
Otherwise, fix $f_* \in \eff^{2\beta +1}$ and apply Claim \ref{e-ind-step} above in the case $\de = \de_{2,2\beta+1}$, $\eff = \eff^{2\beta+1}$,
$\eff^{2\beta + 1} \setminus \{ f_* \}$ and $\langle A_u : u \in \lao \rangle = \langle B_{\mx(u)} : u \in \lao \rangle$.
To complete Stage 3, let $\de_\alpha$ be the filter $\de^\prime$ returned by Claim \ref{e-ind-step}, 
and let $\eff^\alpha = \eff^{2\beta + 1} \setminus \{ f_* \}$. As in Stage 2, the inductive conditions are guaranteed by the statement
of that Claim.

\step{Stage X: A remark on bookkeeping.}
Note that once all the elements of the sequence $\langle B_{\mx(u)} : u \in \lao \rangle$ have appeared as elements $C_\beta$ 
in the enumeration at odd successor steps, Condition (4) will be satisfied by definition of $(I, \de, (\gee|\ba) \cup \eff)$-good triple. 
Likewise, since each indexing function (and therefore each potential sequence $\overline{B}$) occurs cofinally often in our master enumeration, 
and the cofinality of the construction is greater than $\lambda$, 
we are justified in Claim \ref{e-ind-step} in only adjusting for those instances of Boolean terms which the filter
already considers to be small.

\step{Stage 4: Finishing the proof.}
Since there is no trouble in carrying out the induction, we finish by letting $\de = \de_{2,2^\lambda} = \bigcup \{ \de_{2,\alpha} : \alpha < 2^\lambda \}$.
This completes the proof. 
\end{proof}

\section{\hspace{2mm}On flexibility} \label{s:flexible}

In this section we give the necessary background for the non-saturation claim in our main theorem.
That is, we leverage our prior work to show that once we have built a filter $\de$ on $\lambda$ so that
$\mcp(I)/\de$ has the $\mu^+$-c.c. for $\mu < \lambda$, no ultrafilter extending $\de$
will saturate any non-simple or non-low theory. (That is, provided it is built by the method of independent families
of functions -- if an appeal to complete ultrafilters is made, the situation changes, see e.g. Malliaris and
Shelah \cite{mm-sh-v2} Remark 4.2.)

The main definition in this section is \emph{flexible filter}, due to Malliaris
\cite{mm-thesis}. Roughly speaking, the definition assigns a natural size to any given regularizing
family and asks that a flexible filter have regularizing families of arbitrarily small 
nonstandard size.
 
\begin{defn} \label{defn-flexible} \emph{(Flexible filters, \cite{mm-thesis})}
Let $\de$ be a regular filter on $I$, $|I| = \lambda \geq \aleph_0$, and let
$X = \langle X_i : i<\mu \rangle$ be a $\mu$-regularizing family for $\de$.
Say that an element $n_* \in {^I\mathbb{N}}$ is $\de$-nonstandard 
if for each $n\in \mathbb{N}$,
$\{ t \in I : (\mathbb{N}, <) \models n_*[t] > n \} \in \de$.

Define the size of $X$,
$\sigma_X$ to be the element of ${^I \mathbb{N}}$ defined by:
\[  \sigma_X[t] = | \{ i<\mu : t \in X_i \}| ~~ \mbox{for each $t \in I$} \]
\noindent Say that $\de$ is \emph{$\mu$-flexible} if for every $\de$-nonstandard element $n_*$
there is a $\mu$-regularizing family $X \subseteq \de$ such that $\sigma_X \leq n_* \mod \de$. Otherwise,
say that $\de$ is \emph{$\mu$-inflexible} (or simply: \emph{not $\mu$-flexible}). 
When $\mu=\lambda$, we will often omit it. 
\end{defn}

For more on flexibility, see Malliaris \cite{mm4} and recent work of Malliaris and Shelah \cite{mm-sh-v1}-\cite{mm-sh-v2},
where it is shown that flexible is consistently weaker than good. 

Malliaris \cite{mm-thesis} had shown that flexibility is detected by non-low theories, that is:

\begin{fact} \emph{(Malliaris \cite{mm-thesis} Lemma 1.21)}  
Let $T$ be non-low, let $M \models T$ and suppose that $\de$ is a $\lambda$-regular ultrafilter on $I$, $|I|=\lambda$
which is not flexible. Then $M^I/\de$ is not $\lambda^+$-saturated. 
\end{fact}

By a dichotomy theorem of Shelah, any non-simple theory will have either the tree property of the first kind
($TP_1$, or equivalently $SOP_2$) or else of the second kind ($TP_2$). A consequence of Malliaris' proof of the existence of a 
Keisler-minimum $TP_2$-theory in \cite{mm4} is that:

\begin{fact} \emph{(Malliaris \cite{mm4} Lemma 8.8)} 
Let $\de$ be a regular ultrafilter on $\lambda$. If $\de$ saturates some theory with $TP_2$ then $\de$ must be flexible.
\end{fact}

A consequence of recent work of Malliaris and Shelah on ultrapowers realizing $SOP_2$-types is a complementary result.

\begin{fact} \emph{(rewording of Malliaris and Shelah \cite{treetops} Claim 3.11 for the level of theories)}
Let $\de$ be a regular ultrafilter on $\lambda$. If $\de$ saturates some theory with $SOP_2$ then $\de$ must be flexible. 
\end{fact}

Combining these three facts we obtain:

\begin{concl} \label{not-flexible}
Let $\de$ be a regular ultrafilter on $\lambda$ and suppose $\de$ is not flexible. Let $T$ be a theory which is 
either non-low or non-simple, or both. Then $M^\lambda/\de$ is not $\lambda^+$-saturated. 
\end{concl}

\begin{rmk} \emph{(see \cite{mm-sh-v1} Observation 10.9)} \label{r:flex}
$\de$ is $\lambda$-flexible if and only if whenever $f: \fss(\lambda)\rightarrow \de$ is such that
$(u,v \in \fss(\lambda)) \land (|u|=|v|) \implies f(v) = f(u)$ then $f$ has a multiplicative refinement. 
\end{rmk}

The remaining ingredient is a theorem of Shelah which was stated as a constraint on goodness. However,
the proof proceeds by defining countably many elements $\langle A_n : n < \omega \rangle$ of $\de_*$,
and showing that the function $g: \fss(\mu) \rightarrow \de_*$ given by $g(s) = A_{|s|}$ does not have a
multiplicative refinement. Since this function is uniform in the cardinality of s, the proof shows, albeit anachronistically, a failure of flexibility. 
 
\begin{fact} \emph{(Shelah \cite{Sh:c} Claim VI.3.23 p. 364)} \label{s:fact}
Let $\de$ be a maximal filter modulo which $\gee$ is independent, $\kappa = CC(B(\de))$, for infinitely many
$g \in \gee$ $|\rn(g)|>1$ and $\de_* \supseteq \de$ an ultrafilter built by the method of independent families
of functions. Then $\de_*$ is not $\kappa^+$-good.
[More precisely, $\de_*$ is not $\kappa$-flexible.]
\end{fact}

\begin{obs} \label{o:cc}
Let $\mu < \lambda$ and let $\de$ be a regular $\lambda^+$-excellent filter on $\lambda$ given by
Theorem \ref{t:existence} in the case where $\jj(\mcp(I)) = \ba = \ba_{2^\lambda, \mu}$. Then 
$B(\de)$ has the $\mu^+$-c.c. 
\end{obs}

\begin{proof}
Clearly $\ba_{2^\lambda, \mu}$ has the $\mu^+$-c.c.
By definition of $\jj$, whenever $\langle A_i : i < \kappa \rangle$ is a maximal disjoint set of nonzero elements of $B(\de) := \mcp(I)/\de$, we have $\jj(A_i) \neq 0_\ba$, $\jj(A_i) \cap \jj(A_j) = 0_\ba$ for each $i < j < \kappa$ and thus 
$\langle \jj(A_i) : i < \kappa \rangle$ is a pairwise disjoint set of nonzero elements in $\ba$. 
\end{proof}

\begin{cor} \label{f:cor}
Let $\mu < \lambda$ and let $\de$ be a regular $\lambda^+$-excellent filter on $\lambda$ given by
Theorem \ref{t:existence} in the case where $\jj(\mcp(I)) = \ba_{2^\lambda, \mu}$. 
Then no ultrafilter extending $\de$ built by the method of independent functions is $\lambda$-flexible. 
\end{cor}

\begin{proof}
The translation is direct using Observation \ref{o:cc}.
For completeness, we justify compliance with the word ``maximal'' in Fact \ref{s:fact}. 
In the language of Theorem \ref{t:existence}, 
the filter $\de$ is built as the union of an increasing sequence of filters $\de_\alpha$, $\alpha < 2^\lambda$. 
For each $\alpha$, $(I, \de_\alpha, (\gee|\ba) \cup \eff^\alpha)$ is a good triple, and $\eff^{2^\lambda} = \emptyset$. 
Thus, by Fact \ref{uf}, $\de$ is maximal modulo which $\gee|\ba$ remains independent. But by construction, $\gee|\ba$
is isomorphic to $\ba$ and thus to an independent family $\gee^\prime$ of $2^\lambda$ functions each with domain $\lambda$ (or $I$) and range $\mu$.
Letting $\gee^\prime$ stand for $\gee$ in the statement of Fact \ref{s:fact} suffices.
\end{proof}

\begin{concl} \label{c:not-flexible}
Let $\mu < \lambda$ and let $\de$ be the regular $\lambda^+$-excellent filter on $\lambda$ given by
Theorem \ref{t:existence} in the case where $\jj(\mcp(I)) = \ba_{2^\lambda, \mu}$. 
Let $\de_* \supseteq \de$ be any ultrafilter constructed by the method of independent families of functions. 
Then $M^\lambda/\de_*$ is not $\lambda^+$-saturated whenever
$Th(M)$ is non-simple or non-low.
\end{concl}

\begin{proof}
By Corollary \ref{f:cor} and Conclusion \ref{not-flexible}.
\end{proof}

\section{\hspace{2mm}Lemmas for morality} \label{s:rg} 

Section \ref{s:flexible} gave non-saturation by a cardinality argument (small $c.c.$). By Theorem \ref{t:separation}, this shifts the
burden of saturation onto ``morality'' of an ultrafilter on $\ba$. 
Thus, in the next few sections of the paper, our aim is to show that there is an ultrafilter on $\ba = \ba_{2^\lambda,~\mu}$, Definition \ref{d:ba},
which is moral for $T_{rg}$, the theory of the random graph. However, we build a somewhat more general theory.

The key inductive step in constructing that ultrafilter will be to find a multiplicative refinement for each possibility pattern. 
We do this essentially in two stages. First, in this section and the next we show that each such possibility pattern can be covered by $\mu$ approximations each of which has a multiplicative refinement.
Second, we leverage these $\mu$ approximations to produce a multiplicative refinement for the original pattern. 

We first state several results which indicate in what sense the random graph can be seen as
``easier to saturate'' than theories with more dividing. 

\begin{thm-lit} \label{ek-thm} \emph{(Engelking-Kar\l owicz \cite{ek} Theorem 8 p. 284)}
Let $\mu \geq \aleph_0$.
The Cartesian product of not more than $2^\mu$ topological spaces 
each of which contains a dense subset of power $\leq \mu$
contains a dense subset of power $\leq \mu$. 
\end{thm-lit}

\begin{proof} (Sketch) 
Reduce to the case of identifying the dense subsets of the factors with discrete spaces on (at most) $\mu$ elements.
Theorem \ref{thm-iff} above guarantees the existence of an independent family $\eff \subseteq {^\mu}{\mu}$
with $|\eff| = 2^\mu$. Index the Cartesian product $X$ by (a subset of) elements of this family, so $X = \prod_{f \in \eff} X_f$ 
and let the function $\rho: \mu \rightarrow X$ be given by $\eta \mapsto \prod_{f \in \eff} f(\eta)$. 
Then the condition that $\eff$ is independent says precisely that the image of $\rho$ is dense in the
product topology.
\end{proof}

Note that the importance of ``$\mu \geq \aleph_0$'' in Theorem \ref{ek-thm} is for Theorem \ref{thm-iff} and for the conclusion; in particular, 
there is no problem if the dense subsets of the factors are finite. Recall that $T_{rg}$ is the theory of the random graph in the language
$\{ =, R \}$.

\begin{fact} \label{rg-fact}
If $\mu < \lambda \leq 2^\mu$, $A \subseteq \mathfrak{C}_{\trg}$, $|A| \leq \lambda$, 
then for some $B \subsetneq \mathfrak{C}_{\trg}$, $|B|=\mu$ we have that every nonalgebraic
$p \in S(A)$ is finitely realized in $B$.
\end{fact}

\begin{proof}
By quantifier elimination, it suffices to consider $\Delta = \{ xRy, \neg xRy \}$.
Write the Stone space $S_\Delta(A)$ of nonalgebraic types as the product of $\lambda$-many discrete 2-element Stone spaces 
$S_\Delta(\{a \})$ and apply the previous theorem. 
The dense family of size $\mu$ given by the theorem is, in our context, a family of types, 
and since the monster model is $\lambda^+$-saturated, 
we can realize each of them. Call the resulting set of realizations $B$. The hypothesis of density
means precisely that each nonalgebraic type in $S(A)$ is finitely realized in $B$.
\end{proof}

The next few facts simply restate these proofs in a different language. 

\begin{fact} \label{tcf}
If $\mu < \lambda \leq 2^\mu$, we can find a set $A \subset B = \{ \beta : \beta \in {^\lambda 2} \}$ 
such that $|A| \leq \mu$ and $A$ is dense in $B$ in the Tychonoff topology.
\end{fact}

\begin{fact} \label{f:suitable}
Let $\mu < \lambda \leq 2^\mu$, $\ba = \ba_{2^\lambda, ~\mu}$. %$T = \trg$, $\Delta = \{ xRy, \neg xRy \}$.
Let $\ba_0 \subseteq \ba=\ba_{2^\lambda, ~\mu}$ be a Boolean subalgebra generated by $\lambda$ independent 
partitions of $\ba$, see Definition \ref{d:suitable} above,
and let $\ba_1$ be its completion in $\ba$. Then $\ba_1$ can be written as the union 
of $\mu$ ultrafilters.
\end{fact}

\begin{obs} \label{cover}
Let $\mu < \lambda \leq 2^\mu$. 
Let $\ba = \ba_{2^\lambda, ~\mu}$ or $\ba = \mcp(I)/\de$ where $(I, \de, \gee)$ is $(\lambda, \mu)$-good. 
Let $\langle \ma_u : u \in \lao \rangle \subseteq \ba \setminus \{0_\ba\}$ be a  
sequence of elements of $\ba \setminus \{ 0_\ba \}$.
Then there are a complete subalgebra $\ba_1$ of $\ba$ 
and a sequence $\{ \ee_\epsilon : \epsilon < \mu \}$ of ultrafilters of $\ba_1$ such that
\begin{itemize}
\item for each $u \in \lao$, $\ma_u$ is supported by $\ba_1$, i.e. it is based on some partition of $\ba_1$
\item $\ba_1 \setminus \{ 0_\ba \}$ can be written as the union of these $\mu$ ultrafilters
\end{itemize}
In particular, for each $u \in \lao$ there is $\epsilon < \mu$ such that $\ma_u \in \ee_\epsilon$.
\end{obs}

\begin{proof}
By Observation \ref{o:support} and Fact \ref{f:suitable}.
\end{proof}

\begin{defn} \label{qr01} \emph{(The key approximation property)}
\begin{enumerate}
\item[($\qro$)] Let $\qro(T, \vp, \lambda, \mu)$ mean: $T$ is a complete countable first-order theory, $\vp$ is a formula in the language
of $T$, and $\lambda > \mu + |T|$. 
\item[($\qri$)] Let $\qri(T, \vp, \lambda, \mu)$ mean: $\qro(T, \vp, \lambda, \mu)$ and in addition if $(A)$ then $(B)$, where:
\begin{enumerate}
\item[(A)] Given
\begin{enumerate}
\item $\ba = \ba_{2^\lambda, ~\mu}$ as in Definition $\ref{d:ba}$
\item $\de_*$ is an ultrafilter on $\ba$
\item $\overline{\ma} = \langle \ma_u : u \in \lao \rangle$ is a $(\lambda, \ba, T, \vp)$-possibility
with $u \in \lao \implies \ma_u \in \de_*$
\end{enumerate}
\item[(B)] We can find $\langle \uu_\epsilon : \epsilon < \mu \rangle$ and 
$\langle \mb_{u,\epsilon} : u \in [\uu_\epsilon]^{<\aleph_0}, \epsilon < \mu \rangle$ such that:
\begin{enumerate}
\item $\epsilon < \mu \implies \uu_\epsilon \subseteq \lambda$, and $\bigcup \{ \uu_\epsilon : \epsilon < \mu \} = \lambda$
\item $\langle \mb_{u,\epsilon} : u \in [\uu_\epsilon]^{<\aleph_0} \rangle$ is a multiplicative refinement of 
$\overline{\ma}|_{[\uu_\epsilon]^{<\aleph_0}}$, with each $b_{u,\epsilon} \in \ba\setminus \{0_\ba\}$
\item if $\mb \in \de_*$ and $u \in [\lambda]^{<\aleph_0}$ then for some $\epsilon < \mu$ we have $\mb_{u,\epsilon} \land \mb > 0$.
\end{enumerate}
\end{enumerate}

\br
\item[ ] We write $\qri(T, \lambda, \mu)$ to mean that $\qri(T, \vp, \lambda, \mu)$ for all $\vp$ in some critical set
$\mathcal{C}_T$ of formulas for $T$, or alternately for some critical set of possibility patterns, Definition \ref{d:critical}.  
\end{enumerate}
\end{defn}

\begin{obs} \label{o:trans} Suppose that:
\begin{enumerate}
\item $\de$ is a regular, $\lambda^+$-excellent filter on $I$
\item $\de_1$ is an ultrafilter on $I$ extending $\de$
\item $\ba = \ba_{2^\lambda, \mu}$ is a Boolean algebra
\item $\jj : \mcp(I) \rightarrow \ba$ is a surjective homomorphism with $\de = \jj^{-1}(\{ 1_\ba \})$
\item $\vp = \vp(x,y)$ is a formula of $T$.
\item $\de_* = \{ \mb \in \ba : ~\mbox{if} ~\jj(A) = \mb ~\mbox{then}~ A \in \de_1 \}$
\end{enumerate}
If $(A)$ then $(B)$ where:
\begin{enumerate}
\item[(A)] $\qri(T, \vp, \lambda, \mu)$ holds in the case where $\ba$ in Definition \ref{qr01} is replaced by the quotient Boolean algebra $\mcp(I)/\de$.
\item[(B)] $\qri(T, \vp, \lambda, \mu)$. 
\end{enumerate}
\end{obs}

\begin{proof}
By the Transfer Lemma \ref{l:trans}.
\end{proof}

\begin{obs} \label{o:trg}
The set $\{ \vp(x;y,z,w) = (z=w \implies xRy) \land (z\neq w \implies \neg xRy) \}$ is a critical set of formulas for the theory
of the random graph. Moreover, consistency of any set $S$ of instances of $\vp$ follows from consistency of all two-element subsets
of $S$.
\end{obs}

\begin{proof}
By quantifier elimination, since all algebraic types will be automatically realized in regular ultrapowers. 
\end{proof}

\begin{conv}
We will informally write instances of the formula from Observation \ref{o:trg} as 
$\vp(x;a, \trv)$ where $\trv \in \{ 0, 1\}$ is the truth value of $z=w$.
\end{conv}

\begin{lemma} \label{l:sequence}
Let $T$ be the theory of the random graph and $\vp$ the formula from Observation $\ref{o:trg}$. 
Then $\qri(T, \vp, \lambda, \mu)$, thus $\qri(T, \lambda, \mu)$. 
\end{lemma}

\begin{proof}  
Let $\de_*$ be an ultrafilter on $\ba$ and let $\overline{\ma} = \langle \ma_u : u \in \lao \rangle$ be a $(\lambda, \ba, T, \vp)$-possibility satisfying (A) of Definition \ref{qr01}.

It will suffice to prove Observation \ref{o:trans}(A),
thus we work in that setting, i.e. a reduced product where $\de$ is an excellent filter on $I$ and $\mcp(I)$
admits a homomorphism $\jj$ to $\ba$ with $\jj{-1}(\{ 1_\ba \})=\de$. 
Let $M$ be a fixed model of the random graph. 
By the Transfer Theorem, we may associate to $\overline{\ma}$ the (weak) distribution of a nonalgebraic type
\[ p = \{ \vp(x;a_i, \trv_i) : i<\lambda \} \] 
such that the elements $a_i$ belong to ${^IM}$ and $\trv_i \in \{ 0, 1\}$. 
We will give a series of definitions and assertions. 

\step{Step 0: A supporting subalgebra.} 
Apply Observation \ref{cover} to choose a complete subalgebra $\ba_1$ of $\ba$ and 
$\langle \ee^\prime_\epsilon: \epsilon < \mu \rangle$ a covering sequence of ultrafilters.
In what follows, we denote by $X(\ba_1)$ the set $\{ x_{\alpha, \epsilon} : \alpha < \lambda, \epsilon < \mu \}$ of generators of 
$\ba_1$. Note that each element of $\overline{\ma}$ is supported on a partition whose elements
are finite intersections of elements of $X(\ba_1)$. 

Denote by $Y(\ba_1)$ the set of nonempty finite intersections of elements of $X(\ba_1)$, where ``$\my$ is nonempty'' means
$\ba \models$ ``$\my \neq 0$''. This set is the direct analogue of
$\fin_s(\gee)$ in the case where functions $g_\alpha \in \gee$ correspond to $\{ x_{\alpha, \epsilon} : \epsilon < \mu \}$.

\step{Step 1: The collision function $F_\epsilon$.} 
For each $\epsilon < \mu$, define a partial function $F_\epsilon$ from $\lambda$ to $\lambda$ as follows.  
Let $F_\epsilon(i)=j$  if there is some $\mc_h \in Y(\ba_1)$ which witnesses this, which means:
\begin{enumerate}
 \item[($\alpha$)] $\mc_h \in \ee_\epsilon$
 \item[($\beta$)]  $\ba \models$ ``$\{ s \in I : a_i[s] = a_j[s] \} \geq \mc_h$''
 \item[($\gamma$)] if $j_1 < j$ then $\ba \models$ ``$\{ s \in I : a_i[s] = a_{j_1}[s] \} \land \mc_h = 0$''
\end{enumerate}

Note that in condition $(\beta)$ we may ask that ``$\ma_{\{i\}} \land \ma_{\{j\}} \land \{ s \in I : a_i[s] = a_j[s] \} \geq \mc_h$''.
However this is redundant here as $\overline{\ma}$ is a possibility pattern for the random graph, i.e. by choice of $\vp$, 
$\ba \models$ ``$\ma_u = 1$'' whenever $|u|=1$.
 
Note also that for any $\epsilon$ and $i$ there is at most one such $j$. (If not, let $j_1, \mc_{h_1}$ and $j_2, \mc_{h_2}$ be two distinct values
given with their associated witness sets, and notice that $\mc_{h_1} \land \mc_{h_2}$ witnesses both as $\ee_\epsilon$ is a filter, 
contradicting ($\gamma$) by the linear ordering of $\lambda$.) 
Furthermore, $F_\epsilon(i) \leq i$. 

For the remainder of the argument, let $\mc_{h_{\epsilon, i}}$ witness $F_\epsilon(i) = j$. 

\step{Step 2: `Injectivity' of $F_\epsilon$.} if $i_1 \neq i_2$ are from $\dm(F_\epsilon)$ and $j_1 = F_\epsilon(i_1)$, $j_2 = F_\epsilon(i_2)$ then 
\[ \ba \models \left\{ s \in I : a_{i_1}[s] = a_{i_2}[s] \right\} \land \left( \mc_{h_{\epsilon, i_1}} \land \mc_{h_{\epsilon, i_2}} \right) = 0 \]

If not, let $j = \operatorname{min}\{j_1, j_2\}$; then $j, ~\mc_{h_{\epsilon, i_1}} \land \mc_{h_{\epsilon, i_2}}$ contradicts Step 1 for one of the two 
values of $i$. (If $j_1 = j_2$, then $\mc_{h_{\epsilon, i_1}} \land \mc_{h_{\epsilon, i_2}} \in \ee_\epsilon$ contradicts 
$i_1 \neq i_2 \mod \ee_\epsilon$.)

\step{Step 3: The family of approximations, before re-indexing.} Let 
\[ \uu_{\epsilon, \zeta} = \left\{ i<\lambda ~:~ i \in \dm(F_\epsilon) ~\mbox{and}~ \trv_i = f_\zeta (F_\epsilon(i))  \right\} \]

Roughly speaking, we choose only the formulas $\vp(x;a_i)^{\trv_i}$ in the type whose parameter $a_i$ collides modulo $\ee_\epsilon$ 
(as recorded by $F_\epsilon$) with an element whose instance in the type has the same exponent (when filtered through $f_\zeta$,
an element of the dense family of functions). 

In Step 6 we will re-index the double subscript $(\epsilon, \zeta)$ but for now it is a little more transparent to leave it as a pair. 
Note that neither $\ee_\epsilon$ nor any of the sets $\mc_{h_{\epsilon, i}}$ depend on $\zeta$. 

\step{Step 4: The multiplicative refinements.} 
For each $\uu_{\epsilon, \zeta}$ from Step 3, 
and $u \in [\uu_{\epsilon, \zeta}]^{<\aleph_0}$, let
$\mb_{\epsilon, u} = \bigwedge_{i \in u} \mc_{h_{\epsilon, i}}$. 
Let us verify that 
\[ \langle \mb_{\epsilon, u} : u \in [\uu_{\epsilon, \zeta}]^{<\aleph_0} \rangle \] is indeed a multiplicative refinement of 
$\overline{\ma}|_{[\uu_{\epsilon, \zeta}]^{<\aleph_0}}$.

Since all of the $h_{\epsilon, i}$ are from $Y(\ba_1)$ and $h_{\epsilon, i} \in \ee_\epsilon$, 
none of these intersections is empty $\mod \de$ (in fact, they belong to $\ee_\epsilon$) and the assignment is multiplicative. 
Let us verify that it refines the original sequence. By Observation \ref{o:trg}, it suffices to check $|u|\leq 2$. 
The case $|u|=1$ follows by Step 1. Suppose then that $u = \{ i,j \}$. 
By Step 2 and the choice of $\uu_{\epsilon,\zeta}$, whenever $u \in [\uu_{\epsilon, \zeta}]^{<\aleph_0}$ we have that
\[ \ba \models ~\mbox{``}\mb_{u,\epsilon} \land \{ s \in I : a_i[s] = a_j[s], \trv_i \neq \trv_j \} = 0\mbox{''}\]
By choice of $\vp$ and the fact that $\overline{\ma}$ is a possibility pattern, it must be that 
$\ba \models$ ``$\mb_{u,\epsilon} \leq \ma_u$'' as desired. (That is, inconsistency in the random graph can only arise from
equality.)

\step{Step 5: Covering the type.}
In this step we show that the sequence of approximations covers the type, i.e. 
\[ [\lambda]^{<\aleph_0} = \bigcup \{ [\uu_{\epsilon, \zeta}]^{<\aleph_0} : \epsilon < \mu, \zeta < \mu \} \] 
Let $u \in [\lambda]^{<\aleph_0}$ be given, and let $\{ i_\ell : 0 < \ell \leq n \}$ list $u$. 

Informally, we find a set on which the given $n$-tuple is distinct and on which the  
collision function 
is well defined. This set will belong to some $\ee_\epsilon$ by Fact \ref{cover}, and by construction its image under $F_\epsilon$
is as desired. We then need to choose $\zeta$ corresponding to the correct pattern of positive and negative instances,
which we can do by Fact \ref{tcf}. 

More formally, let $\mx_u = \{ s : \bigwedge_{0<\ell<k<n} a_{i_\ell}[s] \neq a_{i_k}[s] \}$ be the set on which all parameters are distinct,
and note that $\mx_u \in \de_*$, so in particular is not $0_\ba$. 
Choose $\my_{h_\ell} \in Y(\ba_1)$, $j_\ell \in \lambda$ by induction on $\ell$, $1 \leq \ell \leq n$ as follows:
\begin{itemize}
\item Choose $h_0$ so that $\my_{h_0} \subseteq \mx_u$ $\mod \de$.
\item For $\ell = m+1 \leq n$, choose $\my_{h_\ell}, j_\ell$ such that $\my_{h_\ell} \leq \my_{h_m}$ and 
\[  \left\{ s : a_{i_m}[s] = a_{j_m}[s] \right\} \geq \mx_{h_\ell} ~~\mbox{and for no $j<j_\ell$ is } 
\{ s : a_{i_m}[s] = a_j[s] \} \land \mx_{h_\ell} \neq 0_\ba \]
\end{itemize}

Let $\epsilon < \mu$ be such that $\my_{h_n} \in \ee_\epsilon$ 
(which we can do by the choice of $\langle \ee_\epsilon : \epsilon < \mu \rangle$). 
Note that $\my_{h_n}$ witnesses $F_\epsilon(i_\ell) = j_\ell$ for $0 < \ell \leq n$. Also, by
choice of $\my_{h_0}$, we have that on $\my_{h_n}$, $\langle j_\ell : 0 < \ell \leq n \rangle$ has no repetitions. So as we chose the
sequence of functions $\langle f_\zeta : \zeta < \mu \rangle$ to be dense, there is $\zeta < \mu$ such that 
\[ \bigwedge_{0<\ell\leq n} f_\zeta(j_\ell) = \trv_{i_\ell} \]  
Now $\uu_{\epsilon, \zeta}$ is as required.

\step{Step 6: Re-indexing.} For notational alignment with Definition \ref{qr01}, re-index this family of triples by $\epsilon < \mu$.

\step{Step 7: Largeness.} Finally, we verify that if $\mb_* \in \de_*$ and $u \in [\lambda]^{<\aleph_0}$ 
then for some $\epsilon < \mu$ we have $\mb_{u,\epsilon} \land \mb_* > 0$.

Note that $\ma_u$ and $\mb_*$ both belong to $\de_*$. 
As $\ba_1$ supports $\overline{\ma}$ and $\de_*$ is an ultrafilter on $\ba_1$, 
we may find $\mx$ such that $\ba_1 \models$ ``$0 < \mx \leq (\ma_u \land \mb_*)$''. Since the sequence $\langle \ee_\epsilon : \epsilon < \mu \rangle$
was chosen to cover $\ba_1 \setminus 0_\ba$, choose $\epsilon_* < \mu$ such that $\mx \in \ee_{\epsilon_*}$. 
Note that we may choose $\epsilon_*$ so that in addition, $u \subseteq \uu_\epsilon$
(since the re-indexing in Step 6 amounted to absorbing the additional parameter $\zeta$).
Then $\mb_{u,\epsilon}$ is nonzero and intersects $\mx$, as they are both members of $\ee_\epsilon$. This completes the proof. 
\end{proof}

\begin{disc} \label{fsoq-discussion}
\emph{Two properties of the random graph which make this proof more transparent are 
first, that the question of whether a given distribution is consistent 
relies only on the pattern of incidence in the parameters, and second, this pattern does not admit too many inconsistencies 
(dividing, or long chains in the Boolean algebra) as described in Fact \ref{rg-fact} and its translation
\ref{tcf}.}

\emph{The class of theories in which consistency of distribution relies only on incidence is quite rich. 
They were studied and classified
in Malliaris \cite{mm5}, where it was shown that any such theory is dominated in the sense of Keisler's order 
either by the empty theory, by the random graph or by the minimum $TP_2$ theory, i.e. the model completion $T^*_{feq}$ of a parametrized
family of crosscutting equivalence relations. (Such theories have intrinsic interest, corresponding naturally to independence properties, and to assertions about
second-order structure on ultrapowers; in fact, the classification applied the second author's proof that there are only 
four second-order quantifiers.) It is clear from this result why ``consistency of distribution relies only on incidence'' 
is not enough to guarantee that the proof of Lemma \ref{l:sequence} goes through. 
In particular, the minimum $TP_2$ theory would allow us to carry out the part of the proof of
Lemma \ref{l:sequence} which had to do with distributing elements so their collisions are controlled by the functions $F_\epsilon$, 
but Fact \ref{rg-fact} would no longer
apply due to the amount of dividing, and the corresponding functions of Fact \ref{tcf} would thus need a larger domain (larger than $\mu$) to
properly code all possible types. Since any ultrafilter which saturates the minimum $TP_2$ theory $T^*_{feq}$ is flexible, as discussed
in \S \ref{s:overview},
the main theorem of this paper shows that the distinction between the random graph-dominated theories and the $T_{feq}$-dominated theories is indeed sharp.}
\end{disc}

\section{The moral ultrafilter} \label{s:morality}

In this section we construct an ultrafilter $\de_*$ on $\ba_{2^\lambda, \mu}$ which is moral for any $\qri$ theory,
and in particular for the theory of the random graph. Note that the random graph is minimum in Keisler's order among the unstable theories,
see \cite{mm-sh-v1} \S 4.

\begin{theorem} \label{t:morality}
Suppose $\mu < \lambda \leq 2^\mu$ and let $\ba = \ba_{2^\lambda,~\mu}$. Then there is an ultrafilter $\de_*$ on $\ba$ which is moral
for all countable theories $T$ such that $\qri(T, \lambda, \mu)$.  In particular, $\de_*$ is moral for all countable stable theories
and for the theory of the random graph. 
\end{theorem}

\begin{proof}
We first prove the ``in particular'' clause. Any such ultrafilter $\de_*$ will be moral for the random graph by Lemma \ref{l:sequence} above.
Moreover, any unstable theory (so in particular the random graph) is strictly above the stable theories in Keisler's order, 
see \cite{Sh:c} Theorem VI.0.3 p. 323. Thus by Theorem \ref{t:separation}, 
any $\de_*$ moral for the random graph must be moral for countable stable theories as well.

In the remainder of the proof, we construct the ultrafilter $\de_*$.

\step{Step 1: Setup for inductive construction of $\de_*$.}
We now build the ultrafilter $\de_*$. 
Enumerate the generators of $\ba$ as $\langle \mx_{\alpha, \epsilon} : \alpha < 2^\lambda, \epsilon < \mu \rangle$  
in the notation of Definition \ref{d:ba}. 
Let $\langle \overline{\ma}_\alpha : \alpha < 2^\lambda \rangle$ be an enumeration of all relevant $(\lambda, \ba, T)$-possibilities,
with each possibility occurring $2^\lambda$-many times. 
[On counting: Note that there are, upto renaming of symbols, at most continuum many complete countable theories, 
so at most continuum many theories such that $\qri(T, \lambda, \mu)$. Moreover,
since we may identify the possibility patterns with sequences from $\fss(\lambda)$ into $\ba$, $|\ba| = 2^\lambda$ 
there are no more than $2^\lambda$ patterns for each theory.]   

We build by induction on $\alpha < 2^\lambda$ a continuous increasing sequence of filters $\de_\alpha$ of $\ba$ and a continuous
decreasing sequence of independent partitions $\gee^\alpha$ satisfying the following conditions. 

\begin{enumerate}
\item $\beta < \alpha < 2^\lambda$ implies $\de_\beta \subseteq \de_\alpha$ are filters on $\ba$
\item $\alpha$ limit implies $\de_\alpha = \bigcup \{ \de_\beta : \beta < \alpha \}$
\item $\beta < \alpha < 2^\lambda$ implies $\gee^\alpha \subseteq \gee^\beta \subseteq \gee$
\item $\alpha < 2^\lambda$ implies $|\gee^\alpha| = 2^\lambda$
\item $\alpha$ limit implies $\gee^\alpha = \bigcap \{ \gee^\beta ~:~ \beta < \alpha \}$
\item $\alpha = \beta + 1$ implies that if $\overline{\ma}_\beta$ is a sequence of elements of $\de_\beta$ then
there is a multiplicative refinement $\overline{\mb}$ of $\overline{\ma}_\beta$ consisting of elements of $\de_\alpha$ 
\item $\alpha < 2^\lambda$ implies that $(\ba, \de_\alpha, \gee^\alpha)$ is a $(2^\lambda, \mu)$-good Boolean triple 
\item $\de_* = \de_{2^\lambda} = \bigcup \{ \de_\alpha : \alpha < 2^\lambda \}$
\end{enumerate}

For the case $\alpha = 0$, let $\de_0 = \{ 1_\ba \}$, $\gee = \{ \{ \mx_{ \alpha, \epsilon } : \epsilon < \mu \} : \alpha < 2^\lambda \}$
be the set of generators of $\ba$ from Definition \ref{d:ba}. 

The limit cases are uniquely determined by the inductive hypotheses (2),(8), and consistent by (the direct translation of) Fact \ref{uf} above. 

\step{Step 2: The successor stage.}
Thus the only nontrivial point is the case $\alpha = \beta+1$. Let $\overline{\ma}_\beta$ be given, and suppose that
$u \in \lao \implies \ma_u \in \de_\beta$; if not, choose $\de_\alpha$ to satisfy condition (4) and continue to the next step.

Noting that we have assumed $u \in \lao \implies \ma_u \in \de_\beta$, let $\de_* \supseteq \de_\beta$ be any ultrafilter on $\ba$
(thus on $\ba_\beta$). Then we may apply Definition \ref{qr01}(A) with
$\overline{\ma}$, $\de_*$, and $\ba_\beta$ in place of $\ba$.

Let $\langle \uu_\epsilon : \epsilon < \mu \rangle$ and $\langle \mb_{u,\epsilon} : u \in [\uu_\epsilon]^{<\aleph_0}, \epsilon < \mu \rangle$
be the objects returned by Definition \ref{qr01}(B). By Definition \ref{qr01}(B)(2), 
for each $\epsilon < \mu$ and $u \in [\uu_\epsilon]^{<\aleph_0}$, we have that $\mb_{u,\epsilon} \in \de^+$, i.e. $\neq \emptyset \mod \de$.

Let $W_\beta = \{\ma_u : u \in \lao \} \cup \{ \mb_{u,\epsilon} : u \in [\uu_\epsilon]^{<\aleph_0}, \epsilon < \mu \}$. 
Since these are all $\de$-nonempty sets, apply Observation \ref{o:support}
to obtain $\gee^\prime \subseteq \gee^\beta$ such that $|\gee^\prime| \leq \lambda$ and 
each element of $W_\beta$ is supported in $\fin_s(\gee^\prime)$.

Let $g_\gamma = \{ \mx_{\gamma, \epsilon} : \epsilon < \mu \}$ be any element of $\gee^\beta \setminus \gee^\prime$.
Let $\gee^\alpha = \gee^\beta \setminus ( \gee^\prime \cup \{ g \} )$.

Now we define the proposed multiplicative refinement. For each $u \in \lao$, define
\[ \mb_u = \bigcup \{ x_{\gamma, \epsilon} \cap \mb_{u,\epsilon} : \epsilon < \mu \} \]
Let $\overline{\mb} = \langle \mb_u : u \in \lao \rangle$. We verify that it is multiplicative:

\begin{align*}
\mb_u \cap \mb_v & = \bigcup \{ \mx_{\gamma, \epsilon} \cap \mb_{u,\epsilon} : \epsilon < \mu \} \cap \bigcup \{ \mx_{\gamma, \epsilon} \cap \mb_{v,\epsilon} : \epsilon < \mu \} \\
 & = \bigcup \{ \mx_{\gamma, \epsilon} \cap (\mb_{u,\epsilon} \cap \mb_{v,\epsilon}) : \epsilon < \mu \} \\
 & ~\mbox{as each approximation is multiplicative} \\
 & = \bigcup \{ \mx_{\gamma, \epsilon} \cap (\mb_{u \cup v,\epsilon}) \} \\
 & = \mb_{u \cup v} \\
\end{align*}

We now show these sets generate a filter which retains enough independence to satisfy (7). 
Let $u \in \lao$ be given. Since $\langle \mb_u : u \in \lao \rangle$ is monotonic, and since, 
as remarked above, each $\mb_u \neq \emptyset \mod \de_\beta$, the set
$\de_\beta \cup \{ \mb_u : u \in \lao \}$ generates a filter which we call $\de^\prime_\alpha$.

Let $A_{h^\prime} \in \fin_s(\gee^\alpha)$. By choice of $\gee^\prime$, there is a nonzero
$A_h \in \fin_s(\gee^\prime)$ such that $\ba \models$ ``$A_h \leq \mb_{u,\epsilon}$''. By the inductive
hypothesis of independence (7), since $\gee^\prime$, $\{ g_\gamma \}$, $\gee^\alpha$ have pairwise empty intersection,
\[ \ba \models \mbox{``} \mb_{u,\epsilon}  \supseteq A_h \cap A_{h^\prime} \cap \mx_{\gamma, \epsilon} \neq 0 \mbox{''} \]
We have shown that $(\ba, \de^\prime_\alpha, \gee^\alpha)$ is a pre-good Boolean triple.
Without loss of generality, extend $\de^\prime_\alpha$ to a filter $\de_\alpha$ so that 
$(\ba, \de_\alpha, \gee^\alpha)$ is a good Boolean triple. This completes the successor stage. 

\step{Step 4: Finish.} 
Note that $\de_*$ will be an ultrafilter, and likewise $\gee^{2^\lambda}$ will be empty by Fact \ref{uf}, as explained at the beginning of 
this section.
This completes the proof.
\end{proof}

\section{The dividing line} \label{s:dividing-line}

\begin{theorem} \label{main-theorem}
Let $\mu < \lambda \leq 2^\mu$. Then there is a regular ultrafilter $\de$ on $\lambda$ such that:
\begin{enumerate}
\item for any countable theory $T$ such that $\qri(\lambda, \mu, T)$ and $M\models T$,
$M^\lambda/\de$ is $\lambda^+$-saturated.
\item in particular, when $T$ is stable or $T$ is the theory of the random graph,
$M^\lambda/\de$ is $\lambda^+$-saturated.
\item for any non-low or non-simple theory $T$ and $M \models T$, $M^\lambda/\de$ is not $\lambda^+$-saturated.
\end{enumerate}
Thus there is a dividing line in Keisler's order among the simple unstable theories. 
\end{theorem}

\begin{proof}
By Theorem \ref{t:separation}, the construction problem separates into a problem of excellence and a problem of morality.

By Theorem \ref{t:existence} there is a $\lambda$-regular, $\lambda^+$-excellent filter $\de_0$ on $\lambda$ 
which admits a surjective homomorphism  $\jj : \mcp(I) \rightarrow \ba = \ba_{2^\lambda, \mu}$ such that $\de_0 = \jj^{-1}(\{ 1_\ba \})$.  

By Theorem \ref{t:morality} there is an ultrafilter $\de_*$ on $\ba_{2^\lambda, \mu}$ 
which is moral for any countable theory $T$ such that $\qri(T, \lambda, \mu)$.

Now let $\de$ be $\{ \jj^{-1}(\mb) : \mb \in \de_* \}$.
Then $\de$ is an ultrafilter and is regular as it extends $\de_0$. 
$\de$ satisfies conditions (1) and (2) by Theorems \ref{t:morality} and \ref{t:separation}, and 
$\de$ satisfies condition (3) by Conclusion \ref{c:not-flexible}.  
This completes the proof. 
\end{proof}

\vspace{10mm}

\section{Appendix: Excellence and goodness} \label{s:appendix}

Here we complete the characterization of good filters as excellent filters.

\begin{rmk}
Though we show here that good implies excellent, 
it is a priori not clear whether natural versions or relatives of goodness (meaning, in our context, ``good for $T$'' or for
families of theories) correspond to the analogous restrictions of excellence, e.g. accuracy of so-called possibility patterns.
\end{rmk}

\begin{claim} \label{c:e-g}
Assume $\de$ is a filter on the Boolean algebra $\ba$. If $\de$ is $\lambda^+$-good then $\de$ is $\lambda^+$-excellent.
\end{claim}

\begin{proof}
Let $\overline{\ma} = \langle \ma_u : u \in \lao \rangle$ be a sequence of elements of $\ba$, and we look for an excellent 
refinement. That is, let $\Lambda = \Lambda_{\ba, \de, \overline{\ma}}$ as in Definition \ref{d:near} above, and
$\ba_1 = \ba/\de$. Then we would like to find $\langle \mb_u : u \in \lao \rangle$ such that whenever 
$\sigma=\sigma(\overline{x}_{\mcp(u)}) \in \Lambda$, i.e. $\sigma(\overline{\ma}^\prime) = 0_{\ba_1}$ for any 
$\overline{\ma}^\prime$ below $\overline{a}$ in the sense of \ref{c:near} (for the Boolean algebra $\ba_1$), 
we have that $\sigma(\overline{\mb}) = 0_{\ba}$.

\br
\step{Step 1: Safe sets.} 
First, for each $u \in \lao$ define:
\[ \ma^1_u = \bigcap \{ 1_\ba - \sigma(\overline{\ma}{\rstr_{\mcp(\upnu)}}) : \upnu \subseteq u, ~\sigma(\overline{x}_{\mcp(\upnu)}) \in \Lambda \} \]

By construction, $\overline{\ma}^1 = \langle \overline{\ma}^1_u : u \in \lao \rangle$ is a monotonic sequence of elements of $\de$. 

\br 
\step{Step 2: A multiplicative refinement.}
Apply the hypothesis of goodness to obtain a multiplicative refinement 
$\overline{\mb}^1 = \langle \mb^1_u : u \in \lao \rangle$  of $\overline{\ma}^1$. Each $\mb^1_u$ is an element of $\ba$, in fact
of $\de$.

\step{Step 3: The sequence $\overline{\mb}$.}
Define $\overline{\mb} = \langle \mb_u : u \in \lao \rangle$ where $\mb_u = \ma_u \cap \mb^1_u$ for each $u \in \lao$.
In the remainder of the proof, we show that $\overline{\mb}$ is the desired excellent refinement of $\overline{\ma}$.
We have immediately from the definition that for each $u \in \lao$,
\begin{itemize}
\item[(a)] $\mb_u \in \ba$
\item[(b)] $\ba \models \mb_u \leq \ma_u$
\item[(c)] $\mb_u = \ma_u \mod \de$
\end{itemize}

It remains to show that excellence holds.

\br
\step{Step 4: Excellence of $\overline{\mb}$.}
For this step, we consider $u \in \lao$ and $\sigma = \sigma(\overline{x}_{\mcp(u)}) \in \Lambda = \Lambda_{\ba, \de, \overline{\ma}}$. 

\br
\step{4a. Remarks.} First, by definition of $\Lambda$ and the fact that $0$ is a constant of the language of Boolean algebras, whenever  
$\sigma = \sigma(\overline{x}_{\mcp(u)}) \in \Lambda = \Lambda_{\ba, \de, \overline{\ma}}$ and
$\overline{\ma}^\prime$ is below $\overline{\ma}\rstr_{\mcp(u)}$ 
with respect to $\ba$ [note: this means substituting in $0_\ba$, not $0_{\ba_1}$, for some elements of $\overline{\ma}^\prime$]
we have that $\ma^1_u \subseteq 1 - \sigma(\overline{\ma}^\prime)$. 

Second, \emph{if} $\upnu \subseteq u$ and $\overline{\ma}^\prime$ is below $\overline{\ma}\rstr_{\mcp(u)}$
in the sense that $\ma^\prime_w = \ma_w$ if $w \subseteq \upnu$ and $\ma^\prime_w = 0_\ba$ otherwise, 
\emph{then} $\ma^1_\upnu\subseteq 1 - \sigma(\overline{\ma}^\prime)$,  just by applying the previous remark twice. 

\step{4b: A partition.}
For each $u \in \lao$ and $\upnu \subseteq u$, define
\[ \mc_{\upnu,u} = \mb^1_\upnu \setminus \bigcup \{ \mb^1_{\upnu \cup \{ t\}} : t \in u \setminus \upnu \}
= \bigcap \{ \mb^1_{t} : t \in \upnu \}  \setminus \bigcup \{ \mb^1_{\{t\}} : t \in u \setminus \upnu \} \]
where the second equality uses multiplicativity of $\overline{\mb}^1$.
Thus $\{ \mc_{\upnu,u} : \upnu \subseteq u \}$ gives a partition of $\mb^1_\emptyset$, 
thus also of $\mb_u \leq \mb^1_u$. 
It will suffice to show that if $\upnu \subseteq u$ and $\mc = \mc_{u,\upnu} > 0$ or $\mc = 1 - \mb^1_\emptyset$, then 
$\ba\rstr_{\mc_{\upnu,u}} \models \sigma(\dots, \mb_w \cap \mc, \dots) _{w \subseteq u}= 0$.

\step{4c. Cases.} 
First, we may justify restricting to $\mb^1_{\emptyset}$ as $\sigma(\dots, 0, \dots) = 0$.

Then letting $\upnu$ vary, we use $\{ \mc_{\upnu,u} : \upnu \subseteq u \}$ to partition $\mb^1_\emptyset$. 
It suffices to show that for each $\mc_{\upnu,u}$, 
$\ba\rstr_{\mc_{\upnu,u}} \models $``$ \sigma(\overline{\mb}{\rstr}_{\mcp(u)}) = 0$''.
Let $u \in \lao$ and $\upnu \subseteq u$ be given. 

If $w \subseteq \upnu$, then $\mc_{\upnu,u} \leq \mb^1_u \leq \mb^1_w$ by definition and by monotonicity. 
Hence $\mb_w \cap \mb^1_w \cap \mc_{\upnu,u} = \mb_w \cap \mc_{\upnu,u} = \ma_w \cap \mb^1_w \cap \mc_{\upnu,u} = \ma_w \cap \mc_{\upnu,u}$, i.e. on $\mc_{\upnu,u}$ we have that $\mb_w = \ma_w$. 

If $w \subseteq u \land w \not\subseteq \upnu$ then $\mb^1_w \cap \mc_{\upnu,u} = 0_\ba$ by definition of $\mc_{\upnu,u}$,
and $\mb_w \leq \mb^1_w$, thus $\mb_w \cap \mc_{\upnu,u} = 0$. 

In other words, writing 
\begin{itemize}
\item $\mb^*_w = \mb_w$ if $\upnu \subseteq w$ and $0_\ba$ otherwise, and
\item $\mb^c_w = \mb_w \cap \mc_{\upnu,u} $ if $\upnu \subseteq w$ and $0_\ba$ otherwise
\end{itemize}
we have shown that
\[ \left( \ba\rstr_{\mc_{\upnu,u}} \models \sigma(\dots \mb^c_w \dots)_{w \subseteq u} = 0 \right) ~~ \Longleftarrow ~~ 
\left( \ba \models \sigma(\dots \mb^*_w \dots)_{w \subseteq u} \cap \mc_{\upnu,u} = 0 \right)   \]

Now by definition of $\mc_{\upnu,u}$, the monotonicity of $\overline{\mb}^1$, and step 4a, we have that 
\[ \mc_{\upnu,u}  \subseteq \bigcap \{ \mb^1_{\{i\}} : i \in \upnu \} \subseteq \mb^1_\upnu \subseteq \ma^1_\upnu \subseteq 1 -  \sigma(\dots \mb^*_w \dots)_{w \subseteq u} \]
which completes the proof.
\end{proof}

\begin{theorem} \emph{(Characterization of goodness)} \label{t:equivalent}
Let $\de$ be a filter on the Boolean algebra $\ba$. Then the following are equivalent.
\begin{enumerate}
\item $\de$ is $\lambda^+$-good.
\item $\de$ is $\lambda^+$-excellent.
\end{enumerate}
\end{theorem}

\begin{proof}
(1) $\rightarrow$ (2) Claim \ref{c:some-examples}, via Remark \ref{r:some-examples}.

(2) $\rightarrow$ (1) Claim \ref{c:e-g}.
\end{proof}

\br
\br
\br

\end{document}